\documentclass[finalsize,finaltitle]{zaa}

\usepackage{psfrag}

\usepackage{graphicx}

\newcommand{\n}{\noindent}
\newcommand{\s}{\smallskip}

\newcommand{\lb}{\linebreak}

\def\tht{\theta}
\def\Om{\Omega}

\def\e{\varepsilon}
\def\g{\gamma}
\def\G{\Gamma}
\def\l{\lambda}
\newcommand{\lambdan}{\lambda}
\newcommand{\psin}{\psi}

\def\p{\partial}

\def\E{\text{e}}
\def\a{\alpha}
\def\b{\beta}

\def\d{\delta}
\def\L{\Lambda}
\def\z{\zeta}

\def\r{\rho}
\def\vp{\varphi}

\def\H{W_2}

\def\Ho{\mathring{W}_2}

\def\di{\,\mathrm{d}}
\def\Th{\Theta}

     \newcommand{\NN}{\mathbb{N}}

     \newcommand{\RR}{\mathbb{R}}
     \newcommand{\ZZ}{\mathbb{Z}}

\DeclareMathOperator{\spec}{\sigma}

\DeclareMathOperator{\dist}{dist}
\DeclareMathOperator{\supp}{supp}

\DeclareMathOperator{\Dom}{\mathcal{D}}

\begin{document}

\allowdisplaybreaks

\title{Spectral Gaps for\\ Self-Adjoint Second Order Operators}

\runtitle{Spectral Gaps for Self-Adjoint Second Order Operators}

\author{Denis Borisov and Ivan Veseli\'c}
\runauthor{D.~Borisov and I.~Veseli\'c}

\address{D.~Borisov: Institute of Mathematics of Ufa Scientific Center of RAS, Chernyshevskogo str.~112, 450008 Ufa, Russia; and
Faculty of Physics and Mathematics, \lb
Bashkir State Pedagogical University, October rev.~st.~3a, 450000 Ufa, Russia; \lb
borisovdi@yandex.ru}
\address{I.~Veseli\'c: Fakult\"at f\"ur Mathematik, TU Chemnitz, 09107 Chemnitz, Germany;  http://www.tu-chemnitz.de/mathematik/stochastik/}

\abstract{%
We consider a second order self-adjoint operator in a domain which can be bounded or unbounded. The boundary is partitioned into two parts with Dirichlet boundary condition on one of them, and  Neumann condition on the other. We assume that the potential part of this operator  is non-negative. We add a localized perturbation assuming that it produces two negative isolated eigenvalues being the two lowest spectral values of the resulting perturbed operator. The main result is a lower bound on the gap between these two eigenvalues. It is  given explicitly in terms of the geometric properties of the domain and the coefficients of the perturbed operator. We apply this estimate to several asymptotic regimes studying its dependence on various parameters. We discuss specific examples of operators to which the bounds can be applied.
}



\primclass{35P15}
\secclasses{35J15}

\received{April 21, 2011; revised February 20, 2012}
\logo{XX}{20XX}{X}{1}{33}
\doi{}

\maketitle

\section{Introduction}
\label{s:Introduction}
Estimates on eigenvalues of lower bounded self-adjoint Hamilton operators are  a classical
object of study in mathematical physics and geometry.
Among them lower bounds for the distance between successive
eigenvalues play an important role. Apart from one-dimensional situations
mostly low lying eigenvalues have been studied in the literature.
This concerns both Schr\"odinger operators \cite{Harrell-78,Simon-84a,Simon-84b,KirschS-85,Simon-85d, KirschS-87,KondejV-06a,Vogt-09}
as well as Laplace operators on general Euclidean domains and manifolds \cite{SingerWYY-85,Yau-03,MaL-08,Yau-09}.
While such questions have been considered already in the eighties \cite{KirschS-85, SingerWYY-85,KirschS-87},
they are attracting the attention of various authors even in recent time \cite{Yau-03,KondejV-06a,MaL-08,Vogt-09,Yau-09}.

Note that already the Perron-Frobenius theory gives one the
information that the lowest eigenvalue for the operators under consideration cannot be degenerate and thus the distance between the lowest two eigenvalues is positive.
This means that lower bounds on this distance are interesting only if one has information on the specific dependence on the parameters of the model
under consideration.

The typical result of this genre gives  a lower bound on the distance between the lowest
and the second lowest eigenvalue in terms of some quantity which is considered as known.
This might be  the potential of the Schr\"odinger operator, or more specifically, the distance between two potential wells.
In geometric situations one may be interested in the dependence of the gap length in terms of the shape of the underlying domain.

We present a lower bound of the distance between the first and second eigenvalue of a selfadjoint second order differential operator
in divergence form on a subdomain of $n$ dimensional Euclidean space. The main features of our result are the following:

\begin{itemize}
 \item The lower bound is explicit in its dependence on the
 coefficients of the differential operator, the potential and the
 geometric data of the domain.
 \item The explicit estimates allow to deduce interesting results
 in various asymptotic regimes studied before.
 \item The result is formulated in terms of a non negative
comparison operator and a localized perturbation. The perturbation
does not need to be necessarily a potential, but may be itself  a
differential operator.
 \item The considered self-adjoint operator and the perturbation are
 quite general, covering a variety of previously considered as well as new examples.
\end{itemize}

While the strategy of the proof of our main result is not completely new,
we need to develop new tools to deal with the more abstract
form of the operator under consideration. This applies in particular for a
 quantitative version of  Harnack's inequality (cf.~Section \ref{s:Quantitative Harnack inequality}) and
lower bounds for positivity regions of eigenfunction derivatives (cf.~Section \ref{s:Proof of main results}).  One more new ingredient is using H\"older continuity of the eigenfunctions and the estimates for their H\"older norms (cf. Lemma~\ref{lm4.1}). It allows us to minimize the restrictions for the smoothness of the eigenfunctions, and therefore, for the coefficients of the studied operators.
In addition, all our estimates are explicit, since we are aiming for a quantitative lower bound
on the spectral gap in the final result.

In the next section we formulate the main result and discuss the consequences in various asymptotic regimes.
It is followed by  Section \ref{s:Examples} devoted to the discussion of examples which are covered by our general model.
Section \ref{s:Preliminaries} establishes some preliminary results about the properties of the quadratic forms of the operators under consideration.
In Section \ref{s:Formula for spectral gaps} we show that a classical formula for the spectral gap holds for our model.
The following section is devoted to a quantitative version of  a Harnack inequality. Section \ref{s:Estimates for 2nd eigenf}
deals with pointwise and $L_p$-estimates for eigenfunctions, and Section \ref{s:Proof of main results}
 concludes the proof of the main theorem.

 \section{Formulation of the problem and the main result}
\label{s:Formulation of the problem and the main result}

We introduce the notation used in the paper and formulate the assumptions for our main theorem.

\subsubsection*{Properties of the quadratic form}
Let $n\geqslant 2$  and $\Omega\subseteq \mathbb{R}^n$ be a connected open set  with $C^1$ boundary.
Let  $0<\nu<\mu< \infty$ and  $A_{ij}\colon \Omega \to \RR $, $i,j \in \{1,\dots, n\}$ be bounded functions such that the ellipticity condition
\begin{equation}\label{1.4}
\begin{split}
&A_{ij}(x)=A_{ji}(x),\quad \nu|\xi|^2\leqslant
\sum_{i,j=1}^{n}A_{ij}(x)\xi_i\xi_j\leqslant
\mu|\xi|^2,
\\
&\text{holds for all  } x=(x_1,\ldots,x_n)\in\mathbb{R}^n, \
\xi=(\xi_1,\ldots,\xi_n)\in \mathbb{R}^n.
\end{split}
\end{equation}
For $V\colon \Omega \to \RR$ set
$V^+(x):=\max\{V(x),0\}$, $V^-(x):=-\min\{V(x),0\}$. We assume
that $V^-\in L_{\frac{q}{2}}(\Om)$ for some $q>n$ and that $V^+\in
L_{\frac{q}{2}}(\Omega')$ for each bounded open set
$\Omega'\subseteq\Omega$. Here we do not exclude  the case
$\p\Omega'\cap\p\Omega\not=\emptyset$.

\s
We introduce the sesquilinear form
\begin{equation}\label{1.5}
\mathfrak{h}[u,v]:=\sum_{i,j=1}^{n} \left(A_{ij}
\frac{\p u}{\p x_i},\frac{\p v}{\p x_j}\right)_{L_2(\Om)}+
(Vu,v)_{L_2(\Om)}
\end{equation}
on $L_2(\Om)$ with the domain
\begin{align*}
\Dom(\mathfrak{h})&:=\Ho^1(\Om,\G)\cap L_2(\Om; V^+),
\\
L_2(\Om; V^+)&:=\left\{u\in L_2(\Om): \int_\Om |u|^2 (V^+ +1)\di
x<\infty\right\},
\end{align*}
where $\G$ is a (possibly empty) subset of the boundary $\p\Om$,    and
$\Ho^1(\Om,\G)$ consists of the functions in $\H^1(\Om)$ vanishing
on $\G$.   This form is symmetric. We will show below (see
Lemma~\ref{lm2.1}) that it is also lower-semibounded and closed.

\subsubsection*{The associated selfadjoint operator}
By $\mathcal{H}$ we denote the self-adjoint operator associated with
the form $\mathfrak{h}$.   We observe that this operator has
Dirichlet boundary conditions on $\G$ and Neumann ones on the remainder of the boundary.

We denote the spectrum of an selfadjoint operator  by the symbol $
\sigma (\cdot)$, by $\lambda_0:= \inf \sigma(\mathcal{H})$ the
infimum of the spectrum of $\mathcal{H}$, and by $\lambda:= \inf
\sigma(\mathcal{H})\setminus \{\lambda_0\}$ the second lowest
spectral value. The aim of this paper is to estimate the spectral
gap between $\l$ and $\lambda_0$, in situations where both of these
numbers are eigenvalues.

\subsubsection*{Geometric assumptions and the comparison operator}
For any $\Omega_0\subseteq \Omega$ denote by $\mathcal{H}_0$ the
self-adjoint operator associated with the same form as in
(\ref{1.5}), but considered on $L_2(\Omega\setminus\overline{\Omega_0})$ with the
domain $\Ho^1(\Omega\setminus\Omega_0,\widetilde{\G})\cap
L_2(\Omega\setminus\Omega_0; V^+)$, where $\widetilde{\G}:=\p\Omega_0\, \cup \,\G$.
This domain corresponds to the Dirichlet condition on $\widetilde{\G}$.

We assume that there exists open and bounded subsets $\Omega_0\subseteq\widehat{\Omega}$ of $ \Omega$ such that
\begin{subequations} \label{cnA1}
\begin{align}\label{2.3-a}
\bullet\;\; &\dist(\Omega_0, \partial\Omega) =d >0, \\
\bullet\;\; &\partial \Omega_0 \text{ is $C^1$-smooth}, \\
\label{2.3-c}
\bullet\;\; &\lambda_0<\l<0\leqslant \inf\spec(\mathcal{H}_0), \\
\label{2.3-d}
\bullet\;\; &V^-\equiv 0\quad\text{on}\quad \Omega\setminus\Omega_0, \\
\label{2.3d}
\bullet\;\; &\dist\{\widehat{\Omega},\p\Omega\}\geqslant \frac{3}{4}d, \\
\label{2.3-e}
\bullet\;\; &\widehat{\Omega} \text{ is path-connected.}
\end{align}
\end{subequations}

We observe that if $\p\Om=\varnothing$, (\ref{2.3-a}) holds true with any $d>0$, and that
(\ref{2.3-c})  excludes the case $V^-\equiv0$, since in this case the operator $\mathcal{H}$
is non-negative.

\begin{figure}[t]
 \begin{center}
 \psfrag{U}{$\Omega$}
 \psfrag{0}{$\Omega_0$}
 \psfrag{H}{$\widehat{\Omega}$}
 \psfrag{x}{$x$}
 \psfrag{y}{$y$}
\centerline{\includegraphics[width=250 true pt]{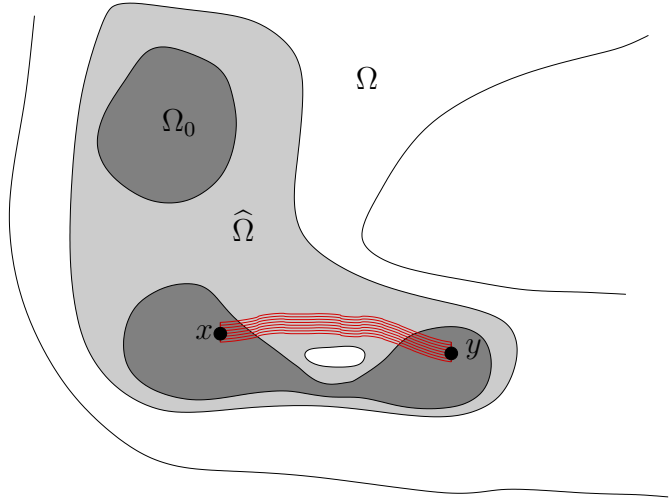}}
 \end{center}
 \caption{Relation between the sets $\Omega_0 \subseteq \widehat{\Omega} \subseteq \Omega$
and an admissible (red)  cylinder.}
 \end{figure}

\subsubsection*{Admissible cylinders}

Given $x,y\in\Omega_0$, consider a $C^2$-curve connecting these
points and lying in $\widehat{\Omega}$. At each point $\hat x$ of
this curve we consider a $(n-1)$-dimensional disk of radius $r$
having $\hat x$ as the center and being orthogonal to the tangential
vector of the curve, where $r$ is a small number. As a result, we
obtain a curved cylinder or tube along the curve. Since the curve is smooth,
we can choose $r$ small enough so that the cylinder does not
overlap with itself. Since the curve is a closed set it has a positive distance
to the boundary of the open set $\widehat{\Omega}$. Thus for $r$ small
enough the cylinder corresponding  to the curve is a subset of
$\widehat{\Omega}$. We call a cylinder with the two mentioned properties
\emph{admissible}.

In Lemma~\ref{lm6.1} we will show that there exist
$L,r_0 \in (0,\infty) $ such that any two points of $\Omega_0$ can be
connected by an admissible cylinder of length at most $L$ and of
radius at least $r_0$. The parameter $L$ plays a role of the linear
size of the domain $\Omega_0$. If $\Omega_0$ is a convex domain, then $L$
is the diameter of $\Omega_0$, and all admissible cylinders can be
chosen straight.

\subsubsection*{Potential strength parameter}
By $\Th_n=\frac{\pi^{\frac{n}{2}}}{\boldsymbol{\G}(\frac{n}{2}+1)}$
and $\tht_n=\frac{2\pi^{\frac{n}{2}}}{\boldsymbol{\G}(\frac{n}{2})}$
we denote the volume of the unit ball and the area of the unit sphere in \vspace{.4mm}
$\mathbb{R}^n$, respectively. Here $\boldsymbol{\G}$ denotes  the gamma function. \vspace{.4mm}
Let $\Omega_{0,t}:=\{x: \dist(x,\Omega_0)<t\}$, \lb
 $\widehat{\Omega}_t:=\{x: \dist(x,\hat\Om)<t\}$,
$B_r(a):=\{x\in \mathbb{R}^n  \mid  |x-a|<r\}$
and
\begin{equation}
  \label{2.0a}
\begin{split}
  \hat{q}&:=\frac{q}{q-2} \\
p &:=\left\{
\begin{aligned}
&\frac{n}{n-2}, &\text{ for }& n>2
\\
&
\hat{q}+1, &\text{ for }& n=2
\end{aligned}
\right.
\\
\widehat{V}&:=\sup_{a\in\widehat{\Omega}_{\frac{d}{4}}}
\|V\|_{L_{\frac{q}{2}}(B_{\frac{d}{4}}(a))} 
\\
&\phantom{:=}\;+\Th_n^{\frac{2}{q}}
\left(\frac{d}{2}\right)^{\!\frac{nq}{2}}
\!\left(
\frac{8\nu}{d^2}+ 3^{\frac{nq}{q-n}} \!\left(\frac{2(p+1)}{\nu}\right)^{\frac{n}{q-n}}
\!\Big(
\sup_{a\in\Om_0}\|V^-\|_{L_{\frac{q}{2}}(B_{\frac{d}{2}}(a))}
\Big)^{\frac{q}{q-n}}\right)\!.
\end{split}
\end{equation}
We observe that $\widehat{V}\not=0$. 

\begin{remark}\label{rm2.1}
In the case $\G=\p\Om$, i.e., once one has the Dirichlet condition on the whole boundary $\p\Om$ for the operator $\mathcal{H}$,
it is possible to replace by zero the term $\frac{8\nu}{d^2}$ in the definition of $\widehat{V}$.
\end{remark}


Now we are in the position to formulate our main result.
It  is a bound for the relative size of the first spectral gap
$\frac{\lambdan-\lambda_0}{|\lambdan|}$.

\begin{theorem}\label{th1.1}
The spectral gap between $\lambdan$ and $\lambda_0$ obeys the following lower bound
\begin{equation}\label{1.6}
\frac{\lambdan-\lambda_0}{|\lambdan|}\geqslant \frac{\Th_{n-1}c_1^{n-1}
\nu}{9L|\Omega_{0,\frac{d}{4}}|\left(\left(\frac{p+1}{\nu} 
\right)^{\frac{n}{q-n}}
\|V^-\|_{L_{\frac{q}{2}}(\Omega_0)}^{\frac{q}{q-n}}+4\mu r_1^{-2}
\right)c_2^{\frac{8L}{d}}},
\end{equation}
\pagebreak
where the constants $c_1=c_1(\mu,\nu,n,q,L,d,r_0,\widehat{V})$,
$c_2=c_2(\mu,\nu,n,q,d,\widehat{V})$ are
defined by
\begin{align*}
c_1&=\min\left\{
r_1 \left(3c_3c_2^{\frac{8L}{d}}\right)^{-\frac{1}{\a}}\!,\frac{d}{8},r_0 \right\},
\\
c_2&=
2^{\frac{p\hat{q}}{p-\hat{q}}+\frac{p\hat{q}+ \hat{q}^2} {(p-\hat{q})^2} -\frac{2(n-1)}{c_7}} \left(\frac{p}{\hat{q}}\right)^{\frac{p^2\hat{q}}
{c_7(p-\hat{q})^2}} \left(1+\frac{2p^2}{\hat{q}^2}\right)^{\frac{\hat{q}} {2(p-\hat{q})}}
\\
&\hphantom{=}\cdot
\max\!\left\{\! (\Th_n d^n)^{\frac{2}{c_7}}\!,\! (\Th_n d^n)^{\frac{p+\hat{q}}{c_7\hat{q}}}\right\}
\max\!\left\{\! c_8^{\frac{\hat{q}} {2(p-\hat{q})}}\!, c_8^{\frac{\hat{q}^2} {2p(p-\hat{q})}}
\!\right\}
\max\!\left\{\!c_9^{\frac{p\hat{q}}{c_7
(p-\hat{q})}}\!, c_9^{\frac{p^2}{c_7
(p-\hat{q})}}
\!\right\}
\\
r_1&=\min\left\{\Th_n^{-\frac{1}{n}} \left(\frac{\nu}{ 12 (p+1)^2
\widehat{V} } \right)^{\frac{q}{2(q-n)}}\!,\frac{d}{4}\right\},
\\
\a&=\min\left\{-\log_4 \left(1-2^{-c_4}\right),1-\frac{n}{q}\right\},
\\
c_3&=4^\a\max\left\{2,
\frac{2^{c_4+2}\nu}{9\sqrt{6}\mu(p+1)\Th_n^{\frac{1}{n}}} \right\},
\\
c_4&=3+81\cdot 2^{n+9} (\tht_n+1)^2n^{-2} \mu^2\nu^{-2}
c_5^{\frac{2(n-1)}{n}},
\\
c_5&=\max\left\{2^{2n+1}\Th_n^{-1},
4^{\frac{q^2n^2}{(q-n)^2}}c_6^{\frac{qn}{q-n}} \right\},
\\
c_6&=9\cdot 2^{2n+9} \Th_n^{\frac{1-q}{q}} n^{-1}
(\theta_n+1)\mu\nu^{-1},
\\
c_7&= \frac{\Th_n^{\frac{1}{2}}\nu^{\frac{1}{2}}}{2^{n+1}\E C_9}\left(
2^{n+4}\mu\Th_n+2^{n\left(1+\frac{2}{q}\right)-3}\widehat{V}
\Th_n^{\frac{1}{\hat{q}}} d^{2(1-\frac{n}{q})} \right)^{-\frac{1}{2}},
\\
c_8&=
\frac{(p+1)^4}{p}\!\left(\!2^{11} \!\left(1+\frac{4\mu}{\nu}\right)\!\frac{\Th_n^{\frac{2}{q}}d^{\frac{2n}{q}-2}}{4^{\frac{2n}{q}}} +4\frac{\widehat{V}}{\nu}  \right)\!
\left(2^{11} \!\left(1+\frac{\mu}{\nu}\right)\frac{\Th_n^{\frac{2}{q}}d^{\frac{2n}{q}-2}}{4^{\frac{2n}{q}}} +\frac{\widehat{V}}{\nu} \right)\!,
\\
c_9&=\frac{4(p+1)^2p\hat{q}\widehat{V}}{(p-\hat{q})^2\nu} \left(1+\frac{c_7}{\hat{q}}\right)
+ 2^{11-\frac{4n}{q}} (p+1)^2\hat{q}
\Th_n^{\frac{2}{q}}d^{\frac{2n}{q}-2}
\\
&\hphantom{=}\cdot
\left(\frac{1}{p}+ \frac{4p}{(p-\hat{q})^2}\frac{\mu}{\nu}\right) \log_{\frac{p}{\hat{q}}}^2
\frac{2p^3}{c_7\hat{q}^2}.
\end{align*}
\end{theorem}

\begin{remark}
\begin{enumerate} 
\item
Note that $|\l|$ is the distance of $\l$ to zero. Hence, the distance between $\l$ and $\lambda_0$ can not be
much smaller than that from $\l$ to zero. In particular, it
means that $\lambda_0$ can not be an accumulation point for the
eigenvalues of $\mathcal{H}$.
\item
Let us note that if there exists  a further eigenvalue
$\widetilde{\l}\in (\l, 0)$ we could give a lower bound similar
to (\ref{1.6}) for  the distance $\widetilde{\l}-\lambda_0$. It turns
out however, that the trivial comparison
$\widetilde{\l}-\lambda_0\geqslant \l-\lambda_0$ gives us a better
estimate.
\item The estimate (\ref{1.6}) is invariant under the multiplication of the operator $\mathcal{H}$ by a constant. 
\end{enumerate}
\end{remark}

One of the main advantages of Theorem~\ref{th1.1} is that the
size of the lower bound is given explicitly.  Although the
formulae look quite bulky, it is possible to study
effectively the dependence of the right hand side in (\ref{1.6})
w.r.t. to various parameters.
This is demonstrated in the next theorems.

\begin{theorem}\label{th2.1}
For $L$ large enough the spectral gap between $\l$ and $\lambda_0$
satisfies the estimate
\begin{equation}\label{1.7}
\frac{\lambdan-\lambda_0}{|\lambdan|} \geqslant c_{10}  \,  L^{-n-1}\E^{-c_{11}L},
\end{equation}
where $c_{10}>0$ depends on $\mu$, $\nu$, $n$, $q$, $d$, $\widehat{V}$, $\|V^-\|_{L_{\frac{q}{2}}(\Om)}$,
and $c_{11}>0$ depends on $\mu$, $\nu$, $n$, $q$, $d$, $\widehat{V}$.
\end{theorem}

Here the phrase ``for $L$ large enough'' should be understood as:
if we keep all parameters except $L$ fixed, there exists some $L_0$,
depending on these other parameters, such that for $L\geqslant L_0$
the claimed estimate holds true.

As it was said above, the parameter $L$ characterizes the linear
size of the domain $\Omega_0$.   Theorem~\ref{th2.5} shows how our
estimate depends on $L$; it turns out that the dependence is very
simple. We also observe that the estimate is exponentially small as
$L\to+\infty$.

\begin{theorem}\label{th2.5}
For $\widehat{V}$ small enough  the spectral gap between $\l$ and $\lambda_0$ satisfies the estimate
\begin{equation}\label{1.11}
\frac{\lambdan-\lambda_0}{|\lambdan|} \geqslant c_{12}  \, L^{-n-1} \E^{-c_{13}L},
\end{equation}
where $c_{12}\!>\!0$ depends on $n$, $q$, $\mu$, $\nu$, $d$, $r_0$,
and $c_{13}\!>\!0$ depends on $n$, $q$, $\mu$, $\nu$, $d$.
\end{theorem}

This theorem addresses the case of  a small potential, which can be
considered as a weak coupling regime.   As it is well known, in this case the
eigenvalues, if they exist, are close to the threshold of the
unperturbed spectrum, in our case, to zero. This fact
is reflected by the estimate (\ref{1.11}), since $|\l|$ tends to
zero as $\widehat{V}\to+0$.

\begin{theorem}\label{th2.2}
For $\nu$ small enough and $\frac{\mu}{\nu}=constant$, the spectral gap
between $\l$ and $\lambda_0$ satisfies the estimate
\begin{equation}\label{1.8}
\frac{\lambdan-\lambda_0}{|\lambdan|}
\geqslant
c_{14} \, L^{-n-1}
(c_{15}\nu)^{c_{16}\frac{L}{\sqrt{\nu}}},
\end{equation}
where $c_{14}>0$ depends on $d$, $n$,   $q$,
$c_{15}>0$ depends on $d$, $n$,   $q$, $\widehat{V}$, $\frac{\mu}{\nu}$,
and $c_{16}>0$ depends on $d$, $n$,   $q$, $\frac{\mu}{\nu}$.
\end{theorem}

This theorem treats the semiclassical regime. In this regime the
functions~$A_{ij}$ read as follows $A_{ij}=\hbar^2 \widetilde{A}_{ij}$, 
where $\hbar\to+0$. Hence, both the parameters~$\mu$ and $\nu$ tend to zero, while $\frac{\mu}{\nu}$ remains constant. It is
known that in certain semiclassical situations the distance between the first two
eigenvalues is exponentially small with respect to $\hbar\approx \sqrt{\nu}$, see e.g., \cite{Simon-84a}.
Theorem~\ref{th2.2} gives a lower bound which
decreases slightly faster than exponentially w.r.t.~$\sqrt{\nu}\to 0$.

\begin{theorem}\label{th2.3}
For $\nu$ small enough and $\mu=constant$, the spectral gap between
$\l$ and $\lambda_0$ satisfies the estimate
\begin{equation}
\frac{\lambdan\!-\!\lambda_0}{|\lambdan|} \geqslant c_{18}  L^{-n-1} \nu^{1+\frac{q(n+1)}{2(q-n)}}
\exp\!\left(\!-c_{19}c_{20} \nu^{-\frac{2(n-1)q}{q-n}}\right)\!
\big(c_{21}\nu^{-1} \log_{\frac{p}{\hat{q}}}\nu\big)^{-c_{22}
c_{20}\frac{L}{\sqrt{\nu}}}\!\!,
\label{1.9}
\end{equation}
where $\displaystyle c_{20}=\exp\Big(c_{23}\nu^{-\frac{2(n-1)q}{q-n}}\Big)$,
$c_{18}=c_{18}(n,d,\mu,\widehat{V})>0$,
$c_{i}=c_{i}(n,q,\mu)>0$ for $i=19,23$,
and $c_{i}=c_{i}(n,q,d,\mu,\widehat{V})>0$  for $i=21,22$.
\end{theorem}

Theorem~\ref{th2.3} is adapted to the models of the photonic
crystals, see Example~3 in Section~\ref{s:Examples}.

\section{Applications}
\label{s:Examples}
In this section we give a series of the examples illustrating
possible applications of our results.

\subsection{Second order differential operator with localized
perturbation}
Let $A_{ij}^{(0)}=A_{ij}^{(0)}(x)$ be bounded
real-valued functions defined on $\Om$ such that
\begin{equation*}
A_{ij}^{(0)}(x)=A_{ji}^{(0)}(x),\quad \nu_0|\xi|^2\leqslant
\sum_{i,j=1}^{n}A_{ij}^{(0)}(x)\xi_i\xi_j\leqslant
\mu_0|\xi|^2,\quad x\in\Omega,
\end{equation*}
where $0<\nu_0<\mu_0$ are constants. By $V^{(0)}$ we denote a
non-negative function defined on $\Om$, so that $V^{(0)}\in
L_{\frac{q}{2}}(\Omega')$ for each bounded domain
$\Omega'\subseteq\overline{\Omega}$. We introduce the self-adjoint
operator $\mathcal{H}^{(0)}$ associated with the form
\begin{equation*}
\mathfrak{h}^{(0)}[u,v]:=\sum_{i,j=1}^{n}
\left(A_{i,j}^{(0)} \frac{\p u}{\p x_i},\frac{\p v}{\p
x_j}\right)_{L_2(\Om)}+ (V^{(0)}u,v)_{L_2(\Om)}
\end{equation*}
on $L_2(\Om)$ with the domain $\Ho^1(\Om,\G)\cap L_2(\Om; V^+)$,
where $\G$ has the same meaning as above. Let $\Omega_0$ be a
bounded subdomain of $\Om$ with $C^1$-boundary and separated
from $\p\Om$ by a positive distance, and
$A_{ij}^{(1)}=A_{ij}^{(1)}(x)$, $V^{(1)}=V^{(1)}(x)$ be bounded
real-valued functions defined on $\Om$ with supports in $\Omega_0$,
and such that the functions $A_{ij}:=A_{ij}^{(0)}+A_{ij}^{(1)}$,
$V:=V^{(0)}+V^{(1)}$ satisfy the conditions~(\ref{1.4}),
 $V^-\in L_{\frac{q}{2}}(\Om)$ for some $q>n$, and $V^+\in
L_{\frac{q}{2}}(\Omega')$ for each bounded domain
$\Omega'\subseteq\overline{\Omega}$.
Then we can consider the operator $\mathcal{H}$ defined via the coefficients
$A_{ij}$ and $V$.

 We also suppose that $\lambda_0=\inf\spec(\mathcal{H})<0$ and $\l<0$ \vspace{.4mm}
are two isolated eigenvalues of $\mathcal{H}$. Then the domain
$\Omega_0$ satisfies the assumption~(\ref{cnA1}). Indeed, in this case
the operator $\mathcal{H}_0$ is determined only by $A_{ij}^{(0)}$
and $V^{(0)}$ and is independent of $A_{ij}^{(1)}$, $V^{(1)}$. It
remains to suppose that  the domain $\Omega_0$ satisfies the \vspace{.4mm}
assumptions~(\ref{2.3d}) and (\ref{2.3-e}), and then one can apply
Theorems~\ref{th1.1}--\ref{th2.3} to the operator~$\mathcal{H}$.

\s
A particular choice of the domain $\Om$ would be a
waveguide-type domain, as depicted in Figure~\ref{fg7.1}. It allows us to apply
our results to such domains which are of interest in the physical
theory of quantum waveguides. One more possible choice is a perforated domain with a
perturbation localized on a bounded subdomain, cf.~Figure~\ref{fg7.2}.

\begin{figure}[th]

\centerline{\includegraphics[width=248 true pt, height=191 true pt]{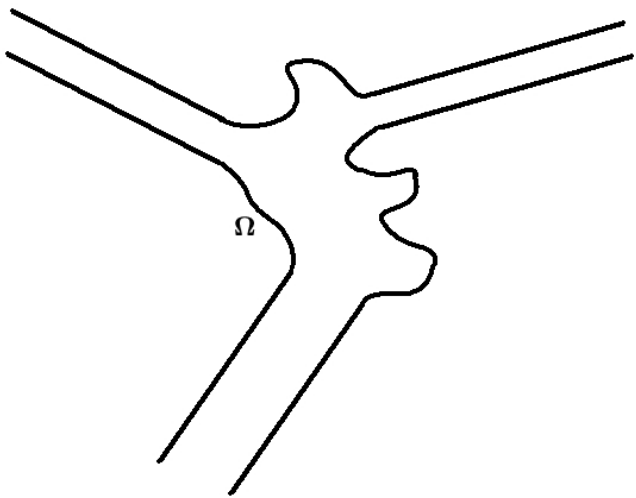}}

\caption{Waveguide}\label{fg7.1}
\end{figure}

\subsection{Perturbation by a singular surface measure}

Although the results of the paper are formulated only for usual
differential operators, they can be  applied to a more general class
of operators. Namely, if some operator can be transformed by a
unitary transformation to the above considered differential
operator, one can apply our results to such operators as well. An
example for such an operator would be the negative Laplacian plus a
singular potential which equals the Hausdorff measure supported on a
manifold of codimension one. Such operators can be transformed to
the above considered differential operators, see Example~5 in
\cite{Borisov-07}.
\bigskip

\subsection{Photonic crystals} Photonic crystals are periodic
dielectric media with the property that the electromagnetic waves
with certain frequencies cannot propagate in it. To achieve this
property one uses high contrast materials. An instance would be the
case where the dielectric constant of the material takes on two
positive values whose quotient is very large. We refer for more
details to the review \cite{Kuchment-01} and the references therein.

\begin{figure}[th]

\centerline{\includegraphics[width=214 true pt, height=172 true
pt]{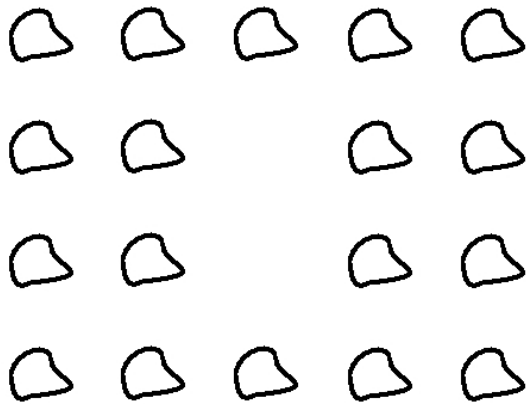}}

\caption{Perforated domain}\label{fg7.2}
\end{figure}

Mathematical models which describe photonic crystals are elliptic
differential operators. The high contrast properties of the medium
are described by the coefficients of the differental operator. In
terms of our notation these properties can be formulated as follows. \vspace{-.4mm}
The set $\Omega$ equals $\RR^n$, the coefficients $A_{ij}$ and~$V$
are introduced in the same way as in Example~1. The functions
$A_{ij}^{(0)}$ are bounded, periodic with respect to the lattice
$\ZZ^n$, and each of them takes on only two values $ 0 < \nu_{ij}\ll
\mu_{ij}< \infty$.

The potential $V^{(0)}$ is $\ZZ^n$-periodic as well. The coefficients
$A_{ij}^{(1)}$, $V^{(1)}$ of the perturbation are bounded, compactly
supported functions. In this  case we think of $\mu$ as a fixed positive
real, while $\nu$ is positive, but very close to zero. Hence, one can
apply in this case Theorem~\ref{th1.1}, or Theorem~\ref{th2.3}.

\subsection{Distant perturbations} Distant perturbations are
perturbations which are localized on a finite number of bounded
domains with large distances between them. More precisely, let
$\Om=\mathbb{R}^n$, and let $Q_k\subset \mathbb{R}^n$ be a finite
number of bounded domains such that the distance between $Q_k$ and
$Q_m$ for  $k\not= m$ is at least $l >0$. Define the operator \vspace{-.6mm}
$\mathcal{H}^{(0)}$ and its coefficients $A_{ij}^{(0)}$, $V^{(0)}$
as in Example~1.  Assume that the coefficients $A_{ij}^{(1)}$,
$V^{(1)}$ are supported in $\bigcup_k Q_k$. The resulting
differential operator $\mathcal{H}$ describes a model for distant or
separated perturbations. Such models were formulated and studied in
their most general form in the references \cite{Borisov-07} and
\cite{Borisov-07a}. There one can also find a review of earlier
results.

Our theorems apply to such problems, too. For $\Omega_0$ one chooses
the union $\bigcup_k Q_k$. Then the main feature of $\Omega_0$ is that
the parameter $L>l$ is large, since $l$ is assumed large. In this
case Theorem~\ref{th2.1} says that the spectral gaps are bounded
from below by an exponentially small quantity w.r.t. $L$. Such
situations have been studied in previous literature, see for
instance, \cite{Harrell-78,KirschS-85,KondejV-06a,Borisov-07,Borisov-07a}. Let us discuss the results of
the two most recent papers in more detail. The main results of the
\cite{Borisov-07,Borisov-07a} are asymptotic expansions for
the eigenvalues of the operators with distant perturbations. These
asymptotics imply that the spectral gap we consider here is
exponentially small w.r.t.~to $L$. In this respect the result of
Theorem~\ref{th2.1} is in a good agreement with that of
\cite{Borisov-07,Borisov-07a}. However, in general the
constants in the estimate in Theorem~\ref{th2.1} are not optimal. In
this respect the results in \cite{Borisov-07,Borisov-07a}
are better than Theorem~\ref{th2.1}.

\section{Preliminaries}
\label{s:Preliminaries} In this section we prove that the operator
$\mathcal{H}$ is well-defined and we study certain properties of its
ground state. We also prove the   existence of bounds $L$ and $r_0$
for the admissible cylinders connecting the points in $\Omega_0$.

In the space $\Dom(\mathfrak{h})$ we introduce the scalar product
given by
\begin{equation*}
(u,v)_{\Dom(\mathfrak{h})}
=(u,v)_{\H^1(\Om)}+(u,v)_{L_2(\Om;V^+)},\quad
(u,v)_{L_2(\Om;V^+)}:=\int_\Om  (1+V^+)u v\di x.
\end{equation*}
>From \cite[Chapter 10, Section 37.2, Theorem 5']{KolmogorovF-75}  it follows that
$\Dom(\mathfrak{h})$ equipped with the scalar product forms a
Hilbert space. 

\begin{lemma}\label{lm2.1}
The form $\mathfrak{h}$ is lower semibounded and closed. The
inequality
\begin{equation}\label{2.4}
\begin{split}
\mathfrak{h}[u,u]&\geqslant -C_1 \|u\|_{L_2(\Om)}^2,
\\
C_1&:= \frac{8\nu}{d^2}+ 3^{\frac{nq}{q-n}} \left(\frac{2(p+1)}{\nu}\right)^{\frac{n}{q-n}}
\left(
\sup_{a\in\Om_0}\|V^-\|_{L_{\frac{q}{2}}(B_{\frac{d}{2}}(a))}
\right)^{\frac{q}{q-n}},
\end{split}
\end{equation}
holds true.
\end{lemma}

\begin{proof}
Let us prove first that the form is lower semibounded and closed. \vspace{.6mm}
Given a point $a\in\Om_0$, consider a ball $B_{\frac{d}{2}}(a)$. Let $\z=\z(t)$ be a function equalling one as $t\leqslant \frac{d}{2}$ and $2-2t/d$ as $t>\frac{d}{2}$. \vspace{.4mm}
Then for any function $u\in\Dom(\mathfrak{h})$ we have $\z(|\cdot-a|)u\in\Ho^1(B_d(a))$. It follows from the definition of $V^-$, the H\"older and Young inequalities, and \cite[Chapter
I\!I, Section 2, Inequality (2.11)]{LadyzhenskayaU-73} that
\begin{equation}
\begin{aligned}
&\big|(V^-u,u)_{L_2(B_{\frac{d}{2}}(a))}\big|
\\
&\leqslant
\|V^-\|_{L_{\frac{q}{2}}(B_{\frac{d}{2}}(a))} \|u\|_{L_{\frac{2q}{q-2}}(B_{\frac{d}{2}}(a))}^2
\\
&\leqslant (p+1)^{\frac{n}{q}}\|V^-\|_{L_{\frac{q}{2}}(B_{\frac{d}{2}}(a))}\|\nabla
\z u\|_{L_2(B_d(a))}^{\frac{2n}{q}}\|\z u\|_{B_d(a)}^{2
\left(1-\frac{n}{q}\right)}
\\
&\leqslant \frac{(p+1)^{\frac{n}{q}}}{q} \|V^-\|_{L_{\frac{q}{2}}(B_{\frac{d}{2}}(a))}
\!\!\left(\!
 n\e \|\nabla\z  u\|_{L_2(B_d(a))}^2
\!+\! (q-n)\e^{-\frac{n}{q-n}}  \|u\|_{L_2(B_d(a))}^2\!\right)
\\
&\leqslant  \frac{(p+1)^{\frac{n}{q}}}{q} \|V^-\|_{L_{\frac{q}{2}}(B_{\frac{d}{2}}(a))}
\\
&\phantom{\leqslant}\;\cdot \Big(
 2n\e  \|\nabla u\|_{L_2(B_d(a))}^2
 + \big((q-n)\e^{-\frac{n}{q-n}}+8n\e d^{-2}   \big) \|u\|_{L_2(B_d(a))}^2 \Big),
\end{aligned}\label{2.8}
\end{equation}
where $\z=\z(|\cdot-a|)$, and  $\e >0$ is arbitrary. It is clear that the set $\Om_0$ can be covered by a finite set of balls $B_{\frac{d}{2}}(a)$ with $a\in\Om_0$, such that each point of $\Om_0$ belongs at most to $3^n$ such balls. Then by the last estimate we have
\begin{align*}
\big|(V^-u,u)_{L_2(\Om_0)}\big|\!&=\!\sum_{a}
\big|(V^-u,u)_{L_2(B_{\frac{d}{2}}(a))}\big| \\
 &\leqslant
3^n\frac{ (p\!+\!1)^{ \frac{n}{q}}}{q} \sup_{a\in\Om_0}\|V^-\|_{L_{\frac{q}{2}}(B_{\frac{d}{2}}(a))}
\\
&\phantom{\leqslant}\;\cdot
\!\Big(\!
  2n\e \|\nabla u\|_{L_2(\Om}^2
   \!+\!\big((q\!-\!n)\e^{-\frac{n}{q-n}}\!+\!8 n\e d^{-2} \big) \|u\|_{L_2(B_d(a))}^2\!\Big).
\end{align*}
We choose
\begin{equation*}
\e:=\frac{q\nu}{2\cdot3^n n(p+1)^{\frac{n}{q}}\sup_{a\in\Om_0} \|V^-\|_{L_{\frac{q}{2}}(B_{\frac{d}{2}}(a))} }
\end{equation*}
and obtain
\begin{gather*}
\big|(V^-u,u)_{L_2(\Om)}\big| \leqslant  \nu \|\nabla
u\|_{L_2(\Omega_0)}^2+\e_1\|u\|_{L_2(\Omega_0)}^2,
\\
\e_1:= \frac{8\nu}{d^2}+ 3^{\frac{nq}{q-n}}\left(1-\frac{n}{q}\right) \left(\frac{2n}{q\nu}\right)^{\frac{n}{q-n}}
(p+1)^{\frac{n}{q-n}} \Big(
\sup_{a\in\Om_0}\|V^-\|_{L_{\frac{q}{2}}(B_{\frac{d}{2}}(a))}
\Big)^{\frac{q}{q-n}}.
\end{gather*}
This inequality and the definition of the form $\mathfrak{h}$
imply that this form is lower-semibounded. Substituting the
obtained inequalities into the definition of $\mathfrak{h}$, taking
into account (\ref{2.3-d}), and applying the estimate
$
t^{\frac{t}{1-t}}(1-t)\leqslant 1,$
$t\in[0,1]$
with $t=\frac{n}{q}$,
we arrive at (\ref{2.4}). It also follows from (\ref{2.8}) with
\begin{equation*}
\e:=\frac{q\nu}{4\cdot3^n n(p+1)^{\frac{n}{q}}\sup_{a\in\Om_0} \|V^-\|_{L_{\frac{q}{2}}(B_{\frac{d}{2}}(a))} }
\end{equation*}
that
$
\frac{\nu}{2}\|\nabla u\|_{L_2(\Om)}^2-C\|u\|_{L_2(\Om)}^2+
\|u\|_{L_2(\Om;V^+)}^2\leqslant \mathfrak{h}(u,u) \leqslant
\mu\|\nabla\|_{L_2(\Om)}^2+\|u\|_{L_2(\Om;V^+)}^2.
$
Employing these inequalities, one can check by the definition
that the form $\mathfrak{h}$ is closed.
\qed\end{proof}

\begin{remark}\label{rm4.1}
In the case of the Dirichlet condition on $\p\Om$ for the operator $\mathcal{H}$ ($\G=\p\Om$) the term $\frac{8\nu}{d^2}$ in the definition of $C_1$  can be replaced by zero. Indeed, it appeared just due to the using of the cut-off function $\z$ in the proof. The presence of this function gave the possibility of applying the results of  \cite[Chapter
I\!I, Section 2, Inequality (2.11)]{LadyzhenskayaU-73}. Once we have the Dirichlet condition on $\p\Om$, it is possible to these results directly to the function $u$ without using $\z$.
\end{remark}

\begin{lemma}\label{lm2.2}
There exists an eigenfunction $\psi_0$ of $\mathcal{H}$
associated with $\lambda_0$ such that $\psi_0(x)>0$ for all $x\in\Om$.
For any compact $\Omega' \subset \Omega$ there exist constants
$C, \tilde C \in (0,\infty)$ such that
\[
 C \le \inf_{\Omega'} \psi_0 \le  \sup_{\Omega'} \psi_0 \le \tilde C .
\]
\end{lemma}

\begin{proof}
By the assumption, $\lambda_0$ is an eigenvalue and consequently there
exists a $\psi_0\in \mathcal{D} (\mathcal{H})$ with
$\mathcal{H}\psi_0 = \lambda_0\psi_0$. Then $|\psi_0|$ is still in the
domain of $\mathfrak{h}$ and is a weak solution of
\begin{equation}\label{2.5}
\mathcal{H}|\psi_0|= \lambda_0|\psi_0|,
\end{equation}
since
$\mathfrak{h}[|\psi_0|,|\psi_0|]=\mathfrak{h}[\psi_0,\psi_0]=\lambda_0$.
Moreover, $\lambda_0|\psi_0|<0$. In view of equation~(\ref{2.5})
we can apply the Harnack inequality to $|\psi_0|$ (see \cite[Section 8.8]{GilbargT-77}),
which shows that on any $\Omega' \Subset
\Omega$ there is are  uniform positive lower and upper bounds on~$\psi_0$.
\qed\end{proof}

\begin{lemma}\label{lm2.3}
The set of the functions in ${C}^\infty(\Om)$ vanishing in a
neighbourhood of $\G$ is dense in $\mathcal{D}(\mathfrak{h})$ in
the topology induced by $(\cdot,\cdot)_{\Dom(\mathfrak{h})}$.

\end{lemma}

This lemma follows from Theorems~1.8.1~and~1.8.2 in
\cite{Davies-89}.

\begin{lemma}\label{lm6.1}
There exists two positive numbers $L$ and $r_0$ such that any two
points in $\overline{\Omega}_0$ can be connected by an admissible
cylinder of the length at most $L$ and of the radius at least $r_0$.
\end{lemma}

\begin{proof}
We fix a point $y\in\Omega_0$ and introduce the sets $\Xi_{m,k}$,
$m,k\in\mathbb{N}$, consisting of all $x\in\widehat{\Omega}$ such
that $x$ can be connected with $y$ by \vspace{.6mm}
an admissible cylinder of
the length less than $m$ and the radius less than $\frac{1}{k}$. One can
see easily that~$\Xi_{m,k}$ are open sets, and
$\Xi_{m,k}\subset\Xi_{\widetilde{m},\widetilde{k}}$ if
$m<\widetilde{m}$, $k<\widetilde{k}$. By the  \vspace{-.4mm}
assumption~(\ref{2.3d}) we conclude that for each
$x\in\widehat{\Omega}$ there exists a curve connecting $x$ and $y$  \vspace{.6mm}
lying in the open set $\hat\Omega$. Due to
(\ref{2.3-e}) it is possible to chose a
sufficiently small radius such that the corresponding cylinder connecting $x$ and $y$ \vspace{.6mm}
is admissible. Therefore,
$\overline{\Omega}_0\subseteq\bigcup_{m,k}\Xi_{m,k}$. The set
$\overline{\Omega}_0$ being compact, we conclude that there exists
a finite cover of $\overline{\Omega}_0$ by the sets $\Xi_{m,n}$. In \vspace{.6mm}
view of monotonicity of these sets w.r.t. $m$, $n$ it implies
that there exists $K$ such that \vspace{.6mm}
$\overline{\Omega}_0\subset\Xi_{K,K}$. Hence, each point in~$\Omega_0$
can be connected with $y$ by an admissible cylinder of the \vspace{.6mm}
length at most~$K$ and of the radius at least $\frac{1}{K}$.

Let us prove that any two points $x_1,x_2\in\Omega_0$ can be
connected by an admissible cylinder of the length at most $3K$
and the radius at least $\frac{1}{4K}$. It is true, if they can
connected by an admissible cylinder of the length at most~$K$
and the radius at least $\frac{1}{K}$. If not, we connect them with $y$
by admissible cylinder of the length at most $K$ and of the
radius at least $\frac{1}{K}$. As a result, we have a cylinder
connecting $x_1$, $x_2$, having the length $2K$ and the radius
$\frac{1}{K}$. Denote this cylinder by $T$. The corresponding curve
connecting $x_1$ and $x_2$ is piecewise $C^2$-smooth, with
possible non-smoothness at $y$. It is clear that we can replace
by a $C^2$-smooth curve of the length at most $3K$ so that there
exists an admissible cylinder corresponding to this curve,
having the radius $\frac{1}{4K}$ and lying \lb
inside~$T$.~\qed
\end{proof}

\section{A formula for spectral gaps}
\label{s:Formula for spectral gaps}
In this section we use a ground state transformation to
establish a formula  for the lowest spectral
gap. For this purpose we will need the next lemma
which follows directly from \cite[Chapter
I\!I\!I, Section 13, Theorem~13.1]{LadyzhenskayaU-73} and \cite[Chapter I\!I\!I,
Section 14, Theorem~14.1]{LadyzhenskayaU-73}.

\begin{lemma}\label{lm3.1}
Let $\lambdan$  be an eigenvalue of $\mathcal{H}$. Then every
eigenfunction $\psin$ associated to $\lambdan$ is continuous in
$\Om$ up to the boundary.
\end{lemma}

Our next aim is to derive a formula for expressions of the form
$\mathfrak{h}[u\psi_0,u\psi_0]$, where $\psi_0$  denotes as before
the ground state. Let $u\in {C}_0^\infty(\overline{\Omega})$.  One can
easily check that $u\psi_0\in \Dom(\mathfrak{h})$, $u^2\psi_0\in
\Dom(\mathfrak{h})$. Hence, $\mathfrak{h}[u\psi_0,u\psi_0]$ is
well-defined. Taking into account the symmetry of $A_{ij}$ (see
(\ref{1.4})) and the definition of $\psi_0$, we check by direct
calculation
$
\sum_{i,j=1}^{n} \Big(A_{i,j} \frac{\p u\psi_0}{\p
x_i},\frac{\p u\psi_0}{\p x_j}\Big)_{L_2(\Om)}=
\sum_{i,j=1}^{n} \Big(A_{i,j} \psi_0\frac{\p u}{\p
x_i},\psi_0\frac{\p u}{\p x_j}\Big)_{L_2(\Om)}$
$+ \sum_{i,j=1}^{n} \Big( \left(A_{i,j} \psi_0\frac{\p
u}{\p x_i},u\frac{\p \psi_0}{\p
x_j}\right)_{L_2(\Om)}+\left(A_{i,j} u\frac{\p \psi_0}{\p
x_i},\frac{\p u\psi_0}{\p x_j}\right)_{L_2(\Om)}\Big)
$
and
\begin{align*}
&\sum_{i,j=1}^{n} \left( \left(A_{i,j} \psi_0\frac{\p
u}{\p x_i},u\frac{\p \psi_0}{\p
x_j}\right)_{L_2(\Om)}+\left(A_{i,j} u\frac{\p \psi_0}{\p
x_i},\frac{\p u\psi_0}{\p
x_j}\right)_{L_2(\Om)}\right)
\\
&=\sum_{i,j=1}^{n} \left( \left(A_{i,j} \frac{\p
\psi_0}{\p x_i},u\psi_0\frac{\p u}{\p x_j
}\right)_{L_2(\Om)}+\left(A_{i,j}\frac{\p \psi_0}{\p x_i},
u\frac{\p u\psi_0}{\p x_j}\right)_{L_2(\Om)}\right)
\\
&=\sum_{i,j=1}^{n}  \left(A_{i,j} \frac{\p \psi_0}{\p
x_i},\frac{\p u^2\psi_0}{\p
x_j}\right)_{L_2(\Om)}
\\
&=\lambda_0\|u\psi_0\|_{L_2(\Om)}^2-
(V\psi_0,u^2\psi_0)_{L_2(\Om)}.
\end{align*}
We substitute these identities into the definition of
$\mathfrak{h}$ to obtain
\begin{equation*}
\mathfrak{h}[u\psi_0,u\psi_0]=\sum_{i,j=1}^{n}
\left(A_{i,j} \psi_0\frac{\p u}{\p x_i},\psi_0\frac{\p u}{\p
x_j}\right)_{L_2(\Om)}+\lambda_0\|u\psi_0\|_{L_2(\Om)}^2.
\end{equation*}
Let $\psi $ be an eigenfunction associated to the eigenvalue $\lambda$.
Let $\Omega'$ be an arbitrary bounded subdomain of $\Om$ separated
from $\p\Om$ by a positive distance. The last relation and
(\ref{1.4}) imply
\begin{equation}\label{3.1}
\begin{aligned}
\frac{\mathfrak{h}[u\psi_0,u\psi_0]-\lambda_0\|u\psi_0\|_{L_2(\Om)}^2}{\|\psin\|_{L_2(\Om)}^2}
&\geqslant \frac{1}{\|\psin\|_{L_2(\Om)}^2}
\sum_{i,j=1}^{n} \left(A_{i,j} \psi_0\frac{\p u}{\p
x_i},\psi_0\frac{\p u}{\p x_j}\right)_{L_2(\Omega')}
\\
&\geqslant \frac{\nu\|\psi_0\nabla
u\|_{L_2(\Omega')}^2}{\|\psin\|_{L_2(\Om)}^2}.
\end{aligned}
\end{equation}
We would like to apply the last formula to $ u =\frac{\psi}{\psi_0}$.
However,  since we do not know whether $\frac{\psi}{\psi_0}$ is in
$u\in {C}_0^\infty(\overline{\Omega})$, we use an approximation argument. \vspace{.8mm}
Let $v_m\in {C}_0^\infty(\Om)$ be a sequence approximating \vspace{.8mm}
$\psin$ in $\Dom(\mathfrak{h})$. Such a sequence exists by
Lemma~\ref{lm2.3}. We take $u=u_m:=\frac{v_m}{\psi_0}$ and pass to the \vspace{-.8mm}
limit as $m\to+\infty$ in (\ref{3.1}). Then the left hand side \vspace{-.8mm}
  converges to
$
\frac{\mathfrak{h}[\psin,\psin]-\lambda_0\|\psin\|_{L_2(\Om)}^2}{\|\psin\|_{L_2(\Om)}^2}
=\lambdan-\lambda_0.
$
Using Lemma~\ref{lm2.2} and  Lemma~\ref{lm2.3} we see that the \vspace{.8mm}
functions $u_m$ converge to~$\frac{\psi}{\psi_0}$ in $\H^1(\Omega')$.
Therefore,
\begin{equation*}
\lambdan-\lambda_0\geqslant
\frac{\nu}{\|\psin\|_{L_2(\Om)}^2}\Big\|\psi_0\nabla
\frac{\psin}{\psi_0}\Big\|_{L_2(\Omega')}^2\geqslant \frac{\nu\,
{\inf_{\Omega'}}\,\psi_0^2}{\|\psin\|_{L_2(\Om)}^2}\Big\|\nabla
\frac{\psin}{\psi_0}\Big\|_{L_2(\Omega')}^2.
\end{equation*}
It follows from the Cauchy-Schwarz   inequality that \vspace{-.4mm}
$
\Big\|\nabla\frac{\psin}{\psi_0}\Big\|_{L_1(\Omega')}^2\!\!\!\!\leqslant \!
|\Omega'|\Big\|\nabla \frac{\psin}{\psi_0}\Big\|_{L_2(\Omega')}^2\!\!.
$
Hence,
\begin{equation}\label{3.2}
\lambdan-\lambda_0\geqslant \frac{\nu\,
{\inf_{\Omega'}}^2\,\psi_0}{|\Omega'|
\|\psin\|_{L_2(\Om)}^2}\Big\|\nabla
\frac{\psin}{\psi_0}\Big\|_{L_1(\Omega')}^2.
\end{equation}
The last identity is the basic formula we use to estimate the
spectral gaps. Before we give the proof of (\ref{1.6}), we need to
establish additional properties of the functions $\psi_0$ and
$\psi$. This is the subject of the next two sections. The proof of
(\ref{1.6}) is given in Section~6.

\pagebreak

\section{Quantitative Harnack inequality}
\label{s:Quantitative Harnack inequality}

In this section we  prove the estimate
\begin{equation}\label{3.5a}
\sup_{\widehat{\Omega}_{\frac{d}{8}}}\psi_0\leqslant
C_2\inf_{\widehat{\Omega}_{\frac{d}{8}}}\psi_0
\end{equation}
with certain $C_2>0$. The constant $C_2$ depends only on the
differential operator under consideration, and   not on the
particular non-negative (sub-) solution; see for instance the
monographs \cite{GilbargT-77,LadyzhenskayaU-73} for a
proof of this statement.

In our context we need to know the explicit dependence of the
constant $C_2$ on the parameters entering the definition of the differential operator.
In the variants of Harnack's  inequality (\ref{3.5a}) given in \cite{GilbargT-77,LadyzhenskayaU-73}
an explicit bound  for $C_2$ is not obtained.
We provide below a proof of Harnack's inequality along the lines of previos arguments,
but which allows explicit control of the constants
 as functions of the various model paramters.

\begin{theorem}
\label{th:Harnack} Assume that $\vp\in \Dom(\mathfrak{h})$ satisfies
the inequalities
\begin{equation} \label{e:subsolution}
(\vp,\phi)_{L_2(\Om)} \ge 0, \quad  \mathfrak{h}[\vp,\phi]\le
\lambdan(\vp,\phi)_{L_2(\Om)}
\end{equation}
for all $\phi \in \Dom(\mathfrak{h})$ which are non-negative
almost everywhere. Then there exists a constant  $C_2$ such that
\begin{equation*}
\sup_{\widehat{\Omega}_{\frac{d}{8}}}\vp\leqslant
C_2\inf_{\widehat{\Omega}_{\frac{d}{8}}}\vp.
\end{equation*}
\end{theorem}
The constant $C_2$ is given explicitly in (\ref{5.16}).
\begin{corollary}
 Let $C >0$ and $q>n$ be fixed. Let $V:\RR^n\rightarrow \RR$ be any potential
such that for all $a \in \RR^n$
\[
 \|V\|_{L_{\frac{q}{2}}(B_1(a)) }\leq C  .
\]
For  $L\in \NN$ and $x \in \RR^n$ denote by $H_\L=-\Delta+V$ the Schr\"odinger operator on $\L:=\Lambda_L(x)=[-\frac{L}{2}, \frac{L}{2}]^n+x$ with periodic boundary conditions,
by $\mathfrak{h}_\L$ the associated quadratic form, and by $\lambda_\L:=\inf\sigma(H_\L)$. Then
there exists a constant  $0< C_H<\infty $, which depends on $q$ and $C$, but not on the particular choice of $V$, nor $L\in \NN$, nor $x \in \RR^n$  such that
for any $\vp\in \Dom(\mathfrak{h}_\L)$  satisfying
\begin{equation*} 
\begin{split}
&(\vp,\phi)_{L_2(\L)} \ge 0, \quad  \text{and} \quad \mathfrak{h}[\vp,\phi]\le
\lambda_\L(\vp,\phi)_{L_2(\L)}
\\
&\text{for all $\phi \in \Dom(\mathfrak{h}_\L)$ which are non-negative
  a.e., }
\end{split}
\end{equation*}
we have
\begin{equation*}
\sup_{\L_1(y)}\vp\leqslant
C_H\inf_{\L_1(y)}\vp \quad
\text{for any unit  box}  \quad \L_1(y) \subset \L_L(x).
\end{equation*}

\end{corollary}

The
remainder of this section is devoted to the proof of the theorem.
The proof of the Harnack inequality in \cite{GilbargT-77} is based on the
Moser-iteration method. To make this iteraton work one needs
first to provide a gradient estimate. This can be derived from
the subsolution property.  We present these arguments which lead
up to the inequalities (\ref{5.3}) and (\ref{5.4}).

Let $\chi\in {C}_0^\infty(\Om)$ be a cut-off function taking values
in $[0,1]$. We introduce an auxilliary quadratic form
\begin{equation*}
\widetilde{\mathfrak{h}}[u,v]:=\sum_{i,j=1}^{n}
(A_{ij}u_i,v_j)_{L_2(\Om)},\quad
\widetilde{\mathfrak{h}}[u]:=\widetilde{\mathfrak{h}}[u,u],
\end{equation*}
where $u=(u_1,\ldots,u_n)$, $v=(v_1,\ldots,v_n)$,
$u_i,v_i\in\Dom(\mathfrak{h})$.  In the second inequality of
(\ref{e:subsolution}) we choose $\phi=\vp^\beta\chi^2$ with
$\b\in\mathbb{R}\setminus\{0\}$ and obtain
\begin{equation*}
\widetilde{\mathfrak{h}}\left[\vp^{\frac{\b-1}{2}} \chi\nabla\vp
\right]=-\frac{2}{\b}
\widetilde{\mathfrak{h}}\left[\vp^{\frac{\b-1}{2}}\chi
\nabla\vp,\vp^{\frac{\b+1}{2}} \nabla\chi\right] +\frac{1}{\b}
\left(W_0\chi\vp^{\frac{\b+1}{2}},\chi\vp^{\frac{\b+1}{2}}
\right)_{L_2(\Om)},
\end{equation*}
where $W_0:=\lambda_0-V$. By the Cauchy-Schwarz inequality and the
symmetry of the matrix $(A_{ij})$  we have
\begin{align*}
\left | 2\widetilde{\mathfrak{h}}\left[\vp^{\frac{\b-1}{2}}\chi
\nabla\vp,\vp^{\frac{\b+1}{2}} \nabla\chi\right] \right|
&\leqslant 2
\left(\widetilde{\mathfrak{h}}\left[\vp^{\frac{\b-1}{2}}\chi
\nabla\vp\right] \right)^{\frac{1}{2}}
\left(\widetilde{\mathfrak{h}}\left[\vp^{\frac{\b+1}{2}}
\nabla\chi\right] \right)^{\frac{1}{2}}
\\
&\leqslant \frac{|\b|}{2}
\widetilde{\mathfrak{h}}\left[\vp^{\frac{\b-1}{2}}\chi
\nabla\vp\right]+\frac{2}{|\b|}
\widetilde{\mathfrak{h}}\left[\vp^{\frac{\b+1}{2}}
\nabla\chi\right].
\end{align*}
The combination of the  two last estimates yields
\begin{equation*}
\widetilde{\mathfrak{h}}\left[\vp^{\frac{\b-1}{2}} \chi\nabla\vp
\right]\leqslant \frac{4}{|\b|^2}
\widetilde{\mathfrak{h}}\left[\vp^{\frac{\b+1}{2}}
\nabla\chi\right]+\frac{2}{|\b|}
\left(|W_0|\chi\vp^{\frac{\b+1}{2}},\chi\vp^{\frac{\b+1}{2}}
\right)_{L_2(\Om)}.
\end{equation*}
It is convenient to introduce the following auxiliary function,
\begin{equation*}
u:=\left\{
\begin{aligned}
&\vp^{\frac{\b+1}{2}}, && \b\not=-1,
\\
&\ln\vp, && \b=-1.
\end{aligned}
\right.
\end{equation*}
Then the last  inequality and (\ref{1.4}) imply
\begin{align}
\|\chi\nabla u\|_{L_2(\Om)}^2&\leqslant C_3(\b) \left(
\|u\nabla\chi\|_{L_2(\Om)}^2+ 
\frac{|\b|}{2\mu}
\left(|W_0|\chi u,\chi u
\right)_{L_2(\Om)} \right),\;\text{if $\b\not=-1$,  } \nonumber
\\
C_3(\b)&:=
\frac{(\b+1)^2}{\b^2}\frac{\mu}{\nu},\label{5.3}
\\
\|\chi\nabla u\|_{L_2(\Om)}^2&\leqslant \frac{4\mu}{\nu}
\|\nabla\chi\|_{L_2(\Om)}^2+\frac{2}{\nu} \left(|W_0|\chi,\chi
\right)_{L_2(\Om)}, \quad \text{ if $\beta=-1$.  }\label{5.4}
\end{align} 
An interpolation inequality for Sobolev spaces (see e.g. \cite[Chapter I\!I, (2.9)]{LadyzhenskayaU-73}) implies
$
\|\chi u\|_{L_{2p}(\Om)}^2 \leqslant (p+1)^2 \|\nabla(\chi
u)\|_{L_2(\Om)}^2.
$
Now estimate (\ref{5.3}) yields
\begin{align*}
\|\chi u \|_{L_{2p}(\Om)}^2&\!\leqslant\! 2 (p\!+\!1)^2
\!\left(\|\chi\nabla
u\|_{L_2(\Om)}^2\!+\!\|u\nabla\chi\|_{L_2(\Om)}^2\right)
\\
&\!\leqslant\! 2(p\!+\!1)^2\!\big(1\!+\!C_3(\b)\!\big)
\|u\nabla\chi\|_{L_2(\Om)}^2
\!+\! \frac{(p\!+\!1)^2 C_3(\b)|\b|}{\mu}
\big(|W_0|\chi
u,\chi u\big)_{L_2(\Om)}
\end{align*}
where, we remind, the number $p$ was introduced in (\ref{2.0a}).
Denote $\hat{q}:=\frac{q}{q-2}$. We employ the H\"older inequality
and arrive at the estimate
\begin{align*}
\|\chi u\|_{L_{2p}(\Om)}^2\leqslant & 2(p+1)^2\big(1+C_3(\b)\big)
|\supp\chi|^{\frac{2}{q}}\|u\nabla\chi\|_{L_{2\hat{q}}(\Om)}^2
\\
&+\frac{(p+1)^2C_3(\b)|\b|}{\mu}\|W_0\|_{L_{\frac{q}{2}}(\supp \chi)}\|\chi
u\|_{L_{2\hat{q}}(\Om)}^2.
\end{align*}
Now we choose the function $\chi$ more specifically. Let
$0<\g_1<\g_2<\frac{1}{4}$, $a\in\widehat{\Omega}$,   and let
\begin{align*}
\chi\equiv 1 \quad\text{in } B_{\g_1 d}(a), \quad \chi\equiv 0
\quad\text{outside } B_{\g_2 d}(a), \quad
\|\nabla\chi\|_\infty\leqslant \frac{2}{(\g_2-\g_1)d}.
\end{align*}
Then
\begin{align}
\|u\|_{L_{2p}(B_{\g_1 d}(a))}^2&\!\leqslant\! C_4(\g_2\!-\!\g_1,\b)
\|u\|_{L_{2\hat{q}}(B_{\g_2 d}(a))}^2,\label{5.6}
\\
C_4(\g_2\!-\!\g_1,\b)&\!:=\! \frac{8
(p\!+\!1)^2\big(1\!+\!C_3(\b)\big)\Th_n^{\frac{2}{q}}d^{\frac{2n}{q}-2}}
{4^{\frac{2n}{q}}(\g_2\!-\!\g_1)^2}\nonumber
\!+\! 
\frac{(p\!+\!1)^2C_3(\b)|\b|}{\mu}
\|W_0\|_{L_{\frac{q}{2}}(B_{\frac{d}{4}}(a))}. \nonumber
\end{align}
Thanks to the abbreviation
\begin{equation*}
\Phi[b,\g]:=\bigg(\int_{B_{\g d}(a)} |\vp|^b\di x
\bigg)^{\frac{1}{b}}\!,\quad b\in \mathbb{R}\setminus\{0\},\quad
\g\in \left(0,\frac{1}{2}\right),
\end{equation*}
inequality (\ref{5.6}) can be rewritten as
\begin{equation}\label{5.7}
\Phi[p(\b+1),\g_1]\leqslant
C_4^{\frac{1}{\b+1}}(\g_2-\g_1,\b)\Phi[\hat{q}(\b+1),\g_2],
\quad \text{if $\b+1>0$,}
\end{equation}
 and
\begin{equation}\label{5.8}
\Phi[\hat{q}(\b+1),\g_2]\leqslant
C_4^{-\frac{1}{\b+1}}(\g_2-\g_1,\b)\Phi[p(\b+1),\g_1], \quad
\text{if $\b+1<0$.}
\end{equation}

Now we can start the iteration procedure mentioned above. For
this purpose we fix a positive number $t$ and choose a sequence
of length scales $\tau_m$ and exponents $\beta_n, m \in \NN$, as
follows
\begin{equation}\label{5.8b}
\tau_m:=\frac{1}{8}+\frac{1}{2^{m+2}}, \quad
\b:=\b_m:=\left(\frac{p}{\hat{q}}\right)^m\frac{t}{p}-1,\quad m \in \NN.
\end{equation}
Recall that by definition $p>\hat{q}$ and choose $t\geqslant
2p$. This ensures  that $\b_m\geqslant 1$ for all $m \in \NN$. The last inequality, 
Lemma~\ref{lm2.1} and the definitions of $p$, $C_3$, $C_4$, and $\widehat{V}$ imply
\begin{equation}\label{5.8a}
\frac{p}{\hat{q}}\leqslant 3,\quad
|\b|\leqslant\frac{3^m t}{p},
\quad 
C_3(\b)\leqslant\frac{4\mu}{\nu},
\quad
C_4(\tau_{m+1}-\tau_{m+2},\b_m)\leqslant 
\frac{4^m t}{p}C_5,
\end{equation}
where
\begin{equation*}
C_5:=
(p+1)^2\left(2^{11} \left(1+\frac{4\mu}{\nu}\right)\frac{\Th_n^{\frac{2}{q}}d^{\frac{2n}{q}-2}}{4^{\frac{2n}{q}}} +4\frac{\widehat{V}}{\nu}  \right).
\end{equation*}
Hence, inequality (\ref{5.7}) yields
\begin{equation}\label{5.9}
\begin{aligned}
\Phi\bigg[t\left(\frac{p}{\hat{q}}\right)^m, \tau_{m+2}
\bigg]&\leqslant \left( 
\frac{4^m C_5 t}{p}
\right)^{\frac{p}{t}
\left(\frac{\hat{q}}{p}\right)^m}\Phi\bigg[t
\left(\frac{p}{\hat{q}}\right)^{m-1}, \tau_{m+1}\bigg]
\\
&\leqslant \prod_{i=1}^m \left( 
\frac{4^i C_5 t}{p}\right)^{\frac{p}{t}
\left(\frac{\hat{q}}{p}\right)^i} \Phi[t,\tau_2],
\end{aligned}
\end{equation}
for all $m\geqslant 1$. Direct calculations show
$
\sum_{i=1}^{\infty} \big(\frac{\hat{q}}{p}\big)^i =
\frac{\hat{q}}{p-\hat{q}}, \
\sum_{i=1}^{\infty}
i\big(\frac{\hat{q}}{p}\big)^i =
\frac{p\hat{q}}{(p-\hat{q})^2}, $
$\max_{[2p,+\infty)} \frac{\ln t}{t}\leqslant
\max_{[2,+\infty)} \frac{\ln t}{t}=\E^{-1}, \
t^{\frac{1}{t}}\leqslant \E^{-\frac{1}{\E}}<2.
$
\vspace{1mm}
We pass in (\ref{5.9}) to the limit $m\to+\infty$.  Then \vspace{.8mm}
\cite[Problem 7.1]{GilbargT-77} and the monotonicity of $\Phi[b,\g] $
with respect to the radius $\g$ imply
\begin{align}\label{5.10}
\sup_{{B_{\frac{d}{8}}(a)}}\vp&=\lim_{m\to+\infty}
\Phi\bigg[t\left(\frac{p}{\hat{q}}\right)^m, \tau_{m+2}
\bigg]\leqslant C_6  \Phi\left[t,\frac{3}{16} \right],
\\
C_6&:=
2^{\frac{p\hat{q}}{p-\hat{q}}+\frac{p\hat{q}}{(p-\hat{q})^2}}
\left(\frac{C_5}{p}\right)^{\frac{p\hat{q}}{t(p-\hat{q})}}.
\end{align}
Thus we are able to bound the supremum of $\vp$ in a small ball
by some $L_p$-norm with a finite exponent on a larger ball.
Similarly, it is possible to give a lower bound on the infimum
of $\vp$ by some $L_p$-norm. For this purpose we consider now
the parameter range $\b+1<0$. In this case two last estimates in
(\ref{5.8a}) remain true. We chose a different sequence of
scales than in  (\ref{5.8b}). More precisely, fix an arbitrary
positive $t$ and set
\[
 \beta   := \beta_m := - \frac{t}{\hat{q}} \left(\frac{p}{\hat{q}}\right)^m -1 , \quad  m \in \ZZ_+. 
\]
We observe that by (\ref{5.6}) and the first estimate in (\ref{5.8a})
$$
C_4(\tau_{m+1}\!-\tau_{m+2},\b_m)\!\leqslant \!4^m\! \!\left(\!\!1\!+\!\frac{t}{\hat{q}}\!\right)\!C_7,
\quad
C_7\!:=\!(p\!+\!1)^2 \!\!\left(\!2^{11} \!\!\left(\!1\!+\!\frac{\mu}{\nu}\right)\!\frac{\Th_n^{\frac{2}{q}}d^{\frac{2n}{q}-2}}{4^{\frac{2n}{q}}} \!+\!\frac{\widehat{V}}{\nu} \! \right)\!\!.
$$

Then estimate (\ref{5.8}) implies
\begin{align*}
\Phi[-t,\tau_1]&\leqslant \left(4^0 C_7\cdot
\left(\frac{t}{\hat{q}}+1\right)\right)^{\frac{\hat{q}}{t}}
\Phi\bigg[-\frac{p t}{\hat{q}},\tau_2\bigg]
\\
&\leqslant \left(4^0 C_7\cdot
\left(\frac{t}{\hat{q}}+1\right)\right)^{\frac{\hat{q}}{t}}
\left(4^1 C_7\cdot \left(\frac{
t}{\hat{q}}+1\right)
\right)^{\frac{\hat{q}^2}{p t}}\Phi\bigg[-\frac{p^2
t}{\hat{q}^2}, \tau_3\bigg]
\\
&\leqslant \prod_{i=1}^m \left(4^{i-1} C_7 
\left(\frac{
t}{\hat{q}}+1\right)  \right)^{\frac{\hat{q}}{t}
\left(\frac{\hat{q}}{p}\right)^{i-1}} \Phi\bigg[-\frac{p^m
t}{\hat{q}^m}, \tau_m\bigg].
\end{align*}
Again we pass  to the limit $m\to+\infty$ and use \cite[Problem
7.1]{GilbargT-77} to arrive at the identity
\begin{equation}
\Phi\left[-t,\frac{1}{4}\right]\leqslant C_8(t) \inf_{{B_{\frac{d}{8}}(a)}}\vp,\quad
C_8(t):=
2^{\frac{\hat{q}^2}{(p-\hat{q})^2}}\left(
C_7\left(\frac{t}{\hat{q}}+1\right)
\right)^{\frac{p\hat{q}}{t(p-\hat{q})}}.
\label{5.12}
\end{equation}

Now we have to cover the intermediate parameter region $0 < b \le
2p$ for the exponent in $\Phi[b,\gamma]$. For this purpose let us
return back to inequality (\ref{5.4}). We fix $a \in \widehat
\Omega$, $\r\leqslant \frac{d}{4}$,  and choose $\chi$ such that
$\chi\equiv1$ in $B_{\r}(a)$, $\chi\equiv0$ outside $B_{2\r}(a)$,
and $\|\nabla\chi\|_\infty\leqslant 2\r^{-1}$. It follows from
(\ref{5.4}) and H\"older inequality that
\begin{align*}
\|\nabla u\|_{L_2(B_\r(a))}^2&\leqslant
2^{n+4}\mu\nu^{-1}\Th_n\r^{n-2}+2^{
1+\frac{n}{\hat{q}}}
\nu^{-1}\Th_n^{\frac{1}{\hat{q}}}\r^{\frac{n}{\hat{q}}}\|W_0\|_{L_{\frac{q}{2}}
(B_{2\r}(a))},
\\
\|\nabla u\|_{L_1(B_\r(a))}^2&\leqslant \Th_n
\r^n\|\nabla u\|_{L_2(B_\r(a))}^2.
\end{align*}
Thus we have established that
$$
\|\nabla u\|_{L_1(B_\r(a))}\leqslant  C_9\r^{n-1},
\quad
C_9:=\Th_n^{\frac{1}{2}} \nu^{-\frac{1}{2}}\!\left(\!
2^{n+4}\mu\Th_n+2^{n\left(1+\frac{2}{q}\right)-3}\widehat{V}
\Th_n^{\frac{1}{\hat{q}}} d^{2(1-
\frac{n}{q})} \!\right)^{\!\frac{1}{2}}\!\!.
$$
If we take any $\widetilde{\rho}\le \rho$ and a ball
$B_{\widetilde{\rho}}(\widetilde{a})$ contained in $\Omega$, then
$
\|\nabla u\|_{L_1(B_\r(a) \cap B_{\widetilde{\rho}}(\widetilde{a}))}
\leqslant C_9\widetilde{\rho}^{n-1}.
$
This shows that the function $|\nabla u|$ is in the Morrey class,
cf.~\cite[Section~7.9]{GilbargT-77}, and that the corresponding norm is
bounded by $C_9$.

Using this estimate, \cite[Chapter 7, Section 7.8, Lemma~7.16]{GilbargT-77},
the inequalities
$
1<\frac{n}{n-1}\leqslant  2,$
$\sum_{j=1}^{\infty} \frac{j^j}{j!(2\E)^j}<\frac{1}{3},
$
and analysing the proof of \cite[Chapter 7, Section~7.9, Lemma~7.20]{GilbargT-77}, one can make sure that
\begin{gather*}
\int_{B_\r(a)}\exp\big(C_{10}|u(x)-u_\r|\big)\di x\leqslant
C_{11}\r^n, \\
u_\r:=\frac{1}{|B_\r(a)|}\int_{B_\r(a)}u(x)\di x, \quad
C_{10}=\frac{\Theta_n}{2^{n+1}\E C_9},\quad C_{11}=2^{n+1}\Th_n.
\end{gather*}
Hence,
\begin{align*}
\int_{B_\r(a)}\exp\big(\pm C_{10}u(x)\big)\di x &\leqslant
C_{11}\r^n\exp(\pm C_{10}u_\r),
\\
\int_{B_\r(a)}\exp\big( C_{10}u(x)\big)\di x
\int_{B_\r(a)}\exp\big(-C_{10}u(x)\big)\di x &\leqslant
C_{11}^2\r^{2n}
\end{align*}
that yields
\begin{equation*}
\int_{B_\r(a)}\exp\big(C_{10}u(x)\big)\di x\leqslant
C_{11}^2\r^{2n} \left(\int_{B_\r(a)}
\exp\big(-C_{10}u(x)\big)\di x\right)^{-1}.
\end{equation*}
We replace $u$ by $\ln\vp$ and obtain
$
\Phi[C_{10},\frac{\r}{d}]\leqslant
\big(C_{11}\r^n\big)^{\frac{2}{C_{10}}}\Phi[-C_{10},\frac{\r}{d}].
$
The relevance of the last estimate is that it relates
$\Phi[b,\gamma]$ with positive  and negative values of $b$ each to
other. The problem is however that we know only that $b=C_{10}$ is
positive, but a close look reveals that it smaller than $2p$, the
  parameter value for which inequality (\ref{5.10}) is valid.
Indeed, $p>1$, and
\begin{equation}\label{4.19a}
C_9\geqslant 2^{\frac{n}{2}+2}\Th_n(\frac{\mu}{\nu})^{\frac{1}{2}}\geqslant
2^{\frac{n}{2}+2}\Th_n, \quad C_{10}< 2^{-\frac{3n}{2}-3}<1.
\end{equation}
For this reason we have to bridge the gap between the parameter
value $C_{10}$ and~$2p$.

We let $\r=\frac{d}{4}$, and use (\ref{5.12}) with $t=C_{10}$. This implies
\begin{equation}\label{5.13}
\Phi\left[C_{10},\frac{1}{4}\right]\leqslant
\left(\frac{C_{11}d^n}{4^n}\right)^{\frac{2}{C_{10}}} C_8(C_{10})
\inf_{{B_{\frac{d}{8}}(a)}}\vp.
\end{equation}
We introduce a sequence
\begin{equation}\label{6.16a}
z_k:= \frac{2\hat{q}^2}{p+\hat{q}}
\left(\frac{\hat{q}}{p}\right)^k,\quad k\geqslant 0.
\end{equation}
Let $l$ be the minimal index in this sequence such that
$z_l\leqslant C_{10}$, i.e., $l$ is the minimal nonnegative integer
greater than or equal to
$-\log_{\frac{p}{\hat{q}}}\frac{C_{10}(p+\hat{q})} {2\hat{q}^2}$. \vspace{-.6mm}
The relations~(\ref{2.0a}) and~(\ref{4.19a}) imply
\begin{alignat*}{3}
\frac{C_{10}(p+\hat{q})}{2\hat{q}^2}&<
\frac{2\hat{q}+1}{2^7\hat{q}^2}&&<\frac{3}{2^7},& n&=2
\\
\frac{C_{10}(p+\hat{q})}{2\hat{q}^2}&<C_{10}p &&<
\frac{n}{2^{\frac{3n}{2}+3}(n-2)}<\frac{3}{2^{\frac{3n}{2}+3}}\leqslant\frac{3}{2^7},&
\quad n&>2.
\end{alignat*}
Hence,
$-\log_{\frac{p}{\hat{q}}}\frac{C_{10}(p+\hat{q})}{2\hat{q}^2}>0$,
and it follows from the definition of $l$ that
\begin{equation}\label{6.1}
l=-\log_{\frac{p}{\hat{q}}}\frac{C_{10}(p+\hat{q})}{2\hat{q}^2}
+\eta,
\end{equation}
where $\eta\in[0,1)$. Thus,
\begin{equation}\label{6.2}
C_{10}\frac{\hat{q}}{p}< z_l\leqslant C_{10}.
\end{equation}
By   H\"older inequality we obtain
\begin{equation}\label{5.14}
\Phi\left[z_l,\frac{1}{4}\right]\leqslant \left(\frac{\Th_n
d^n}{4^n}\right)^{\frac{C_{10}-z_l}{C_{10} z_l}}\Phi\left[C_{10},\frac{1}{4}\right].
\end{equation}

Our next aim is to estimate $\Phi[t,\frac{3}{16}]$ by $\Phi[z_l,\frac{1}{4}]$ for
some $t\geqslant 2p$. This will be again done by an interation, but
this time it will have only a finite number of steps. We introduce
the number $N$ as the minimal integer greater than or equal to
\begin{equation}\label{5.14a}
\log_{\frac{p}{\hat{q}}}\frac{2p}{z_l}=l+1+ \log_{\frac{p}{\hat{q}}}
\frac{p+\hat{q}}{\hat{q}}.
\end{equation}
The identity (\ref{6.1}) implies the upper bound for $N$,
\begin{equation}\label{6.5}
N\leqslant l+2+
\log_{\frac{p}{\hat{q}}}\frac{p+\hat{q}}{\hat{q}} \leqslant
\log_{\frac{p}{\hat{q}}}\frac{p+\hat{q}}{\hat{q}}-\log_{\frac{p}{\hat{q}}}
\frac{C_{10}(p+\hat{q})}{2\hat{q}^2}+3
\leqslant
\log_{\frac{p}{\hat{q}}} \frac{2p^3}{C_{10}\hat{q}^2}.
\end{equation}
It follows from the definition of $z_l$ that
\begin{equation}\label{5.14b}
\max_{m\in
\mathbb{N}}\left|\frac{z_l}{\hat{q}}\left(\frac{p}{\hat{q}}\right)^m-1\right|^{-1}
=\left|\frac{z_l}{\hat{q}}\left(\frac{p}{\hat{q}}\right)^l-1\right|^{-1}
=\frac{p+\hat{q}}{p-\hat{q}}
\end{equation}

This time it turns out to be convenient to choose the
sequences of length scales~$\widetilde{\tau}_m$ and exponents
$\beta_m$, $m\in\NN$, according to
$$
\widetilde{\tau}_m=\frac{1}{4}-\frac{m}{16N}, \quad
\b:=\beta_m :=\frac{z_l}{\hat{q}}  \left(\frac{p}{\hat{q}}\right)^m-1.
$$
By (\ref{4.19a}), (\ref{6.2}), (\ref{6.5}), (\ref{5.14b}), (\ref{5.14a}) and the
definition of $C_3$ and $C_4$ we obtain
\begin{gather*}
C_3(\b)\leqslant \frac{(|\b|+1)^2}{|\b|^2}\frac{\mu}{\nu}=
\left(1+\frac{1}{|\b|}\right)^2\frac{\mu}{\nu} \leqslant \frac{4p^2}{(p-\hat{q})^2}\frac{\mu}{\nu},
\\
C_4(\widetilde{\tau}_m-\widetilde{\tau}_{m-1},\b_m)\leqslant C_{12}\left(\frac{p}{\hat{q}}\right)^m,
\\
\begin{split}
&C_{12}:=\frac{4(p+1)^2p^2}{(p-\hat{q})^2}\frac{\widehat{V}}{\nu} \left(1+\frac{C_{10}}{\hat{q}}\right)
\\
&\phantom{C_{12}:=}\;
+ 2^{11-\frac{4n}{q}}(p+1)^2 \Th_n^{\frac{2}{q}}d^{\frac{2n}{q}-2} \left(1+ \frac{4p^2}{(p-\hat{q})^2}\frac{\mu}{\nu}\right) \log_{\frac{p}{\hat{q}}}^2
\frac{2p^3}{C_{10}\hat{q}^2}.
\end{split}
\end{gather*}

\pagebreak
\n
Taking these relations and (\ref{6.2}) into account, we apply
the estimates (\ref{5.7}),
\begin{align*}
&\Phi\left[z_l\left(\frac{p}{\hat{q}}\right)^N,\frac{3}{16}
\right]=\Phi\left[z_l\left(\frac{p}{\hat{q}}\right)^N,
\widetilde{\tau}_N\right]
\\
&\leqslant \left(C_{12}\left(\frac{p}{\hat{q}}\right)^{N-1}
\right)^{\frac{\hat{q}}{z_l}
\left(\frac{\hat{q}}{p}\right)  ^{N-1}}\Phi\left[z_l
\left(\frac{p}{\hat{q}}\right)^{N-1},
\widetilde{\tau}_{N-1}\right]
\\
&\leqslant
\left(C_{12}\frac{\hat{q}}{p}\left(\frac{p}{\hat{q}}\right)^N
\right)^{\frac{p}{z_l} \left(\frac{\hat{q}}{p}\right)^N}
\left(C_{12}\frac{\hat{q}}{p}\left(\frac{p}{\hat{q}}\right)^{N-1}
\right)^{\frac{p}{z_l} \left(\frac{\hat{q}}{p}\right)^{N-1}}
\Phi\left[z_l \left(\frac{p}{\hat{q}}\right)^{N-2},
\widetilde{\tau}_{N-2}\right]
\\
&\leqslant\prod_{i=1}^N
\left(C_{12}\frac{\hat{q}}{p}\left(\frac{p}{\hat{q}}\right)^i
\right)^{\frac{p}{z_l} \left(\frac{\hat{q}}{p}\right)^i}
\Phi[z_l,\widetilde{\tau}_0] \leqslant
\left(\frac{p}{\hat{q}}\right)^{\frac{p^2\hat{q}}
{z_l(p-\hat{q})^2}}\left(C_{12}\frac{\hat{q}}{p}\right)^{\frac{p\hat{q}}
{z_l(p-\hat{q})}}\Phi\left[z_l,\frac{1}{4}\right]
\\
&\leqslant\left(\frac{p}{\hat{q}}\right)^{\frac{p^2\hat{q}}
{C_{10}(p-\hat{q})^2}}
\left(C_{12}\frac{\hat{q}}{p}\right)^{\frac{p\hat{q}}
{z_l(p-\hat{q})}}
\Phi\left[z_l,\frac{1}{4}\right].
\end{align*}
Note that $z_l(\frac{p}{\hat{q}})^N \ge 2p$ by (\ref{5.14a}). We
choose $t=z_l(\frac{p}{\hat{q}})^N$ and combine the obtained
inequality with (\ref{5.10}), (\ref{5.13}), and (\ref{5.14}),
\begin{gather}
\sup_{{B_{\frac{d}{8}}(a)}}\vp\leqslant C_{13}
\inf_{{B_{\frac{d}{8}}(a)}}\vp, \label{5.15}
\\
\begin{aligned}
C_{13}\!:=&2^{\frac{p\hat{q}}{p-\hat{q}}+\frac{p\hat{q}+
\hat{q}^2}{(p-\hat{q})^2} -\frac{2(n-1)}{C_{10}}} \!\!\left(\!\frac{p}{\hat{q}}\right)^{\!\!\frac{p^2\hat{q}}
{C_{10}(p-\hat{q})^2}}\!\! \left(\!1+\frac{2p^2}{\hat{q}^2}\right)^{\!\frac{\hat{q}} {2(p-\hat{q})}}
\!\!\max\!\left\{\! (\Th_n d^n)^{\frac{2}{C_{10}}}, (\Th_n d^n)^{\frac{p+\hat{q}}{C_{10}\hat{q}}}\!\right\}
\\
&\cdot \!\max\!\left\{\!\! \left(\!\frac{C_5C_7}{p}\!\right)^{\!\frac{\hat{q}} {2(p-\hat{q})}}\!\!, \!\left(\!\frac{C_5C_7} {p}\!\right)^{\!\frac{\hat{q}^2} {2p(p-\hat{q})}}\!\right\}\!
\max\!\left\{\!\!\left(\!\frac{C_{12}\hat{q}}{p} \!\right)^{\!\frac{p\hat{q}}{C_{10}
(p-\hat{q})}}\!\!, \!\left(\!\frac{C_{12}\hat{q}}{p}\!\right)^{\!\frac{p^2}{C_{10}
(p-\hat{q})}}\!\right\}
\end{aligned}
\nonumber
\end{gather}
for all $a\in\widehat{\Omega}$, where we have used that by
(\ref{4.19a}), (\ref{6.2}), (\ref{6.5}), (\ref{6.16a}),  and $t\geqslant 2p$
$
2p\leqslant t=z_l\left(\frac{p}{\hat{q}}\right)^N\leqslant \frac{2\hat{q}^2}{p+\hat{q}} \left( \frac{p}{\hat{q}}\right)^{N-l}\leqslant \frac{2p^2}{\hat{q}},
$
\begin{align*}
\left(\frac{\Th_n
d^n}{4^n}\right)^{\frac{C_{10}-z_l}{C_{10} z_l}} \left(\frac{C_{11}\Th_n
d^n}{4^n}\right)^{\frac{2}{C_{10}}} &= \frac{(\Th_n d^n)^{\frac{1}{C_{10}}+\frac{1}{z_l}}}{2^{\frac{2n}{z_l}-\frac{2}{C_{10}}}} \\
&\leqslant 2^{-\frac{2(n-1)}{C_{10}}} \max\left\{ (\Th_n d^n)^{\frac{2}{C_{10}}}, (\Th_n d^n)^{\frac{p+\hat{q}}{C_{10}\hat{q}}}\right\},
\\
\left(\frac{C_5 C_7}{p}\right)^{\frac{p\hat{q}}{t(p-\hat{q})}}
&\leqslant
\max\left\{ \left(\frac{C_5C_7}{p}\right)^{\frac{\hat{q}} {2(p-\hat{q})}}, \left(\frac{C_5C_7} {p}\right)^{\frac{\hat{q}^2} {2p(p-\hat{q})}}
\right\},
\\
\left(\frac{C_{12}\hat{q}}{p} \right)^{\frac{p\hat{q}}{z_l
(p-\hat{q})}}
&\leqslant
\max\left\{\left(\frac{C_{12}\hat{q}}{p} \right)^{\frac{p\hat{q}}{C_{10}
(p-\hat{q})}}, \left(\frac{C_{12}\hat{q}}{p}\right)^{\frac{p^2}{C_{10}
(p-\hat{q})}}
\right\}.
\end{align*}
We observe that
\begin{equation}\label{5.17}
C_{13}\geqslant 1,
\end{equation}
otherwise inequality (\ref{5.15}) is impossible.

Let $x_{max}\in\widehat{\Omega}_{\frac{d}{8}}$ be the point of global maximum
of $\vp$ in $\widehat{\Omega}_{\frac{d}{8}}$,   and $x_{min}$ be the point of
global minimum in the same domain. Then there exists two points
$a_\pm\in\widehat{\Omega}$ such that $x_{max}\in
\overline{B}_{\frac{d}{8}}(a_+)$, $x_{min}\in\overline{B}_{\frac{d}{8}}(a_-)$. We
connect the points $a_+$ and $a_-$ by an admissible cylinder of the
length at most $L$; the corresponding curve lies in $\widehat{\Omega}$.
We cover the corresponding curve by balls of radius $\frac{d}{8}$; the balls
should have at least common boundary points, and two of these balls
must be $B_{\frac{d}{8}}(a_\pm)$. It is clear that it is possible to cover
the mentioned curve by at most $\frac{4L}{d}$ balls.   We apply estimate
(\ref{5.15}) to each of these balls and proceed as in the proof of
\cite[Chapter I\!I, Section 2.3, Theorem~2.5]{GilbargT-77}, that leads us
to the estimate~(\ref{3.5a}), where
\begin{equation}\label{5.16}
C_2=C_{13}^{\frac{4L}{d}}.
\end{equation}

\section{Estimates for the second eigenfunction $\psin$}
\label{s:Estimates for 2nd eigenf}  In this section we study properties of
an eigenfunction $\psin$ associated to an eigenvalue $\lambda\in
(\lambda_0, 0)$. In particular, we establish a relation between the
supremum-, the H\"older- and $L_2$-norms of $\psi$.

\begin{lemma}\label{lm4.2}
There exist a point $x_-\in\Omega_0$ such that $\psin(x_-)=0$.
\end{lemma}

\begin{proof}
We prove the existence by contradiction. Suppose that such a point~$x_-$ does not exist and consider the domain $\Omega_-:=\{x\in\Om:
\psin(x)\leqslant 0\}$.  Lemma~\ref{lm3.1} yields that this domain
is closed and by the assumption $\Omega_-\cap\Omega_0=\varnothing$. Now we
restrict the quadratic form $\mathfrak{h}$ to the subspace
$\Ho^1(\Omega_-,\G\cap\p\Omega_-)\cap L_2(\Omega_-; V^+)$.   \lb
This form is
closed, symmetric and lower-semibounded. By $\mathcal{H}_-$ we
denote the associated self-adjoint operator in $L_2(\Omega_-)$. It
follows from the identity \lb
$\Omega_-\cap\Omega_0=\varnothing$ and the
definition of the function $\psin$ that it belongs to   \lb
$ \Ho^1(\Omega_-,\G\cap\p\Omega_-)\cap L_2(\Omega_-; V^+)$ and is a generalized
solution to the equation
\begin{equation*}
\left(-\sum_{i,j=1}^{n}\frac{\p}{\p x_i}A_{ij}\frac{\p}{\p
x_j}+V-\lambdan\right)\psin=0\quad \text{in }\Omega_-
\end{equation*}
satisfying Dirichlet boundary condition on $\G\cap\p\Omega_-$, and the
Neumann condition on the rest of the boundary. Therefore, it is an
eigenfunction of this operator associated with an eigenvalue
$\lambdan<0$. By Dirichlet-Neumann bracketing and~(\ref{2.3-c}) it
follows that
$
\inf\spec(\mathcal{H}_-)\geqslant \inf\spec(\mathcal{H}_0)\geqslant
0,
$which yields a contradiction. Therefore,
$\Omega_-\cap\Omega_0\not=\varnothing$, i.e., there exists $x_-\in\Omega_0$
such that $\psin(x_-)=0$.
\qed\end{proof}

Having Lemma~\ref{lm3.1} in mind and, if needed,   changing the sign
of $\psin$, we normalize the function $\psin$ by the requirement
\begin{equation}\label{3.3}
\max_{\overline{\Omega}_0}|\psin|
=\max_{\overline{\Omega}_0}\psin=1.
\end{equation}

 The function $\psi$ is a generalized solution to the equation
\begin{equation*}
\left(-\sum_{i,j=1}^{n}\frac{\p}{\p x_i}A_{ij}\frac{\p}{\p
x_j}+V^+-\lambdan\right)\psin=0\quad \text{in }
\Omega\setminus\Omega_0.
\end{equation*}
Due to (\ref{2.3-c}) we have $V^+-\lambdan\geqslant 0$ in
$\Omega\setminus\Omega_0$. Together with the fact that $\psi$ vanishes on
$\p\Om$, by the weak maximum principle (see \cite[Chapter 8, Section~8.1,
Theorem 8.1]{GilbargT-77}) and (\ref{3.3}) we have the estimate
$
|\psi(x)|\leqslant \max_{\p\Omega_0}|\psi|\leqslant 1,$
$x\in\overline{\Omega\setminus\Omega_0}.
$
By~(\ref{3.3}) it yields
\begin{equation}\label{6.3c}
\max_{\overline{\Omega}}|\psi|=1.
\end{equation}

\begin{lemma}\label{lm4.1}
For each ball $B_r(a)$, $a\in\Omega_0$, $r\leqslant r_1$, the
inequality
\begin{equation}\label{4.1}
|\psin(x)-\psin(y)|\leqslant  \frac{C_{14} r^\a}{r_1^\a},\quad
x,y\in \overline{B}_r(a)
\end{equation}
holds true, where
\begin{align*}
r_1:=\;&\left\{\Th_n^{-\frac{1}{n}} \left(\frac{\nu}{ 12 (p+1)^2
\widehat{V} } \right)^{\frac{q}{2(q-n)}}, \frac{d}{4}\right\},
\\
\a=\;&\min\left\{-\log_4 \left(1-2^{-C_{15}}\right),1-\frac{n}{q}\right\},
\\
C_{14}=\;&4^\a\max\left\{2,
\frac{2^{C_{15}+2}\nu}{9\sqrt{6}\mu(p+1)\Th_n^{\frac{1}{n}}}
\right\},
\\
C_{15}:=\;&3+81\cdot 2^{n+9} (\tht_n+1)^2n^{-2} \mu^2\nu^{-2}
C_{16}^{\frac{2(n-1)}{n}},
\\
C_{16}:=\;&\max\left\{2^{2n+1}\Th_n^{-1},
4^{\frac{q^2n^2}{(q-n)^2}}C_{17}^{\frac{qn}{q-n}} \right\},
\\
C_{17}:=\;&9\cdot 2^{2n+9} \Th_n^{\frac{1-q}{q}} n^{-1}
(\theta_n+1)\mu\nu^{-1}.
\end{align*}
\end{lemma}

\begin{proof}
The statement of this lemma was proven in \cite[Chapter I\!I\!I,
Section 13, Theorem 13.1]{LadyzhenskayaU-73}, but the explicit formulae for $\a$,
$r_1$ and $C_{14}$ were not given. For this reason we partially
reproduce the proof to obtain the explicit formulae for the
mentioned constants. The idea of the proof is to estimate the
norm of the gradient of $\psin$ on some special sets. The
resulting estimates guarantee that $\psin$ belongs to a certain
class of functions which can be embedded into a H\"older space.
The H\"older norm and exponent can be expressed explicitly via
the constants in the estimates for the norm of the gradient.

We choose $\widetilde{a}\in\Omega_0$.  We are going to prove that for
an appropriate choice of the constants $C$ and $\widetilde{C}$ the
function $\psin$ belongs to the special class of functions
$\mathfrak{B}_2(B_{\frac{d}{4}}(\widetilde{a}),1,C,\widetilde{C},1,\frac{1}{q})$
defined in \cite[Chapter I\!I, Section 6]{LadyzhenskayaU-73}. It was shown
in \cite[Chapter I\!I, Theorem 6.1]{LadyzhenskayaU-73} that this class
is embedded in a H\"older space. Moreover,   an explicit formula for
the H\"older exponent and the estimate for the norm   were given.
This is why we need to estimate the constants $C$, $\widetilde{C}$
explicitly to apply the cited theorem.

Given any number $k\geqslant -2$ and any ball $B_r(a)\subset
B_{\frac{d}{4}}(\widetilde{a})$,
we consider the set
$M_{k,r}:=\{x: \psin(x)>k\}\cap B_r(a)$. We only need to consider
$k\geqslant -2$, since this is a superset  of the values of $k$
which satisfy inequality~(6.2) in \cite[Chapter
I\!I]{LadyzhenskayaU-73}, namely
\[
 k \geqslant \max_{x\in B_r(a) } \pm \psin (x) -1.
\]
We note that  this inequality forms  part of the definition of the functional class
$\mathfrak{B}_2(B_{\frac{d}{4}}(\widetilde{a}),1,C,\widetilde{C},1,\frac{1}{q})$.

Let $\chi\in C^\infty(\overline{B_r(a)})$ be a cut-off function
taking values in $(0,1)$ and vanishing outside $B_r(a)$. The
function $v(x):=\chi(x)^2\max\{\psin(x)-k;0\}$ belongs to
$\Dom(\mathfrak{h})$ and since $\psin$ is a weak solution we
have
$
\mathfrak{h}[\psin,v]-\lambdan(\psin,v)_{L_2(\Om)}=0.
$
We substitute this relation and the definition of $v$into the
formula for $\mathfrak{h}$,
\begin{align*}
\sum_{i,j=1}^{n}\int_{ M_{k,r}}     A_{ij}\frac{\p
\psin}{\p x_i} \frac{\p \psin}{\p x_j}\chi^2 \di x = &-2
 \sum_{i,j=1}^{n} \int_{ M_{k,r}}     A_{ij}\frac{\p
\psin}{\p x_i} \frac{\p \chi}{\p x_j}(\psin-k)\chi\di x
\\ &
\,+  \int_{ M_{k,r}}     W(\psin-k)\psin\chi^2 \di x,
\end{align*}
where $W:=\lambdan-V$. Employing (\ref{1.4}) and Cauchy-Schwarz
  inequality, we continue the calculations,
\begin{align}
\nu\|\chi\nabla \psin\|_{L_2(M_{k,r})}^2  \leqslant&
\frac{\nu}{2}\|\chi\nabla
\psin\|_{L_2(M_{k,r})}^2+2\mu^2\nu^{-1}
\|(\psin - k)\nabla\chi\|_{L_2(M_{k,r})}^2\nonumber
\\ &
 +   \int_{ M_{k,r}}     |W| |\psin|(\psin - k)\chi^2 \di x,\nonumber
\\*[.5mm]
\|\chi\nabla \psin\|_{L_2(M_{k,r})}^2  \leqslant& 4\mu^2\nu^{-2}
\|(\psin - k)\nabla\chi\|_{L_2(M_{k,r})}^2 \label{4.2}
\\ &
 +  2\nu^{-1}   \int_{ M_{k,r}}     |W| |\psin|(\psin - k)\chi^2 \di x.
\nonumber
\end{align}
\pagebreak
To estimate the last term in this inequality, we  apply the H\"older
inequality and Lemma~\ref{lm2.1},
\begin{align*}
\int_{M_{k,r}} |W|  |\psin| (\psin-k)\chi^2\di x &\leqslant
\|W\|_{L_{\frac{q}{2}}(M_{k,r})} \|\psin |\psin-k|
\chi^2\|_{L_{\hat{q}}(M_{k,r})}
\\
&\leqslant \|W\|_{L_{\frac{q}{2}}(M_{k,r})}
\|(|\psi-k|^2+k|\psi-k|)\chi^2\|_{L_{\hat{q}}(M_{k,r})}
\\
&\leqslant  \|W\|_{L_{\frac{q}{2}}(B_{\frac{d}{4}}(\widetilde{a}))}
\|3|\psin-k|^2\chi^2+ 2k^2\|_{L_{\hat{q}}(M_{k,r})}
\\
&\leqslant
\widehat{V} \Big(3\|(\psin-k)\chi\|_{L_{2\hat{q}}(M_{k,r})}^2+
2k^2|M_{k,r}|^{\frac{1}{\hat{q}}}\Big).
\end{align*}
Now we employ inequality \cite[Chapter I I, Section
2, (2.12)]{LadyzhenskayaU-73} that implies
\begin{equation}\label{6.3a}
\begin{split}
      \|(\psin - k)\chi\|_{L_{2\hat{q}}( M_{k,r} )}^2&  \leqslant
{C}_4|M_{k,r}|^{2(\frac{1}{n}-\frac{1}{q})}   \|\nabla
(\psin - vk)\chi\|_{L_2( M_{k,r} )}^2
\\
&  \leqslant  {C}_4
\Th_n^{2(\frac{1}{n}-\frac{1}{q})}
r^{2(1-\frac{n}{q})}
\\ &\phantom{\leqslant}\;\cdot
  \left(  \|\chi\nabla
\psin\|_{L_2( M_{k,r} )}^2 +  \|
(\psin - vk)\nabla\chi\|_{L_2( M_{k,r} )}^2  \right) .
\end{split}
\end{equation}
We substitute two last inequalities into (\ref{4.2}) and take into \vspace{-.8mm}
account that the definition of $r_1$ and the inequality $r\leqslant
r_1$ imply \vspace{.6mm}
$
3{C}_4 \Theta_n^{2\left(\frac{1}{n}-\frac{1}{q}\right)}
r^{2\left(1-\frac{n}{q}\right)} \nu^{-1}\widehat{V} \leqslant \frac{1}{4}.
$
 We also bear in mind that since $M_{k,r}=\varnothing$ for $k>1$, \vspace{.6mm}
we can restrict our consideration to the case \vspace{.6mm}
$-2\leqslant k\leqslant 1$. In this case $|k|\leqslant 2$. By (\ref{4.2}),
(\ref{6.3a}) it leads us to the estimate
\begin{align*}
\|\chi\nabla \psin\|_{L_2(M_{k,r})}^2 &\leqslant  (8\mu^2\nu^{-2}+1)
\|(\psin-k)\nabla\chi\|_{L_2(M_{k,r})}^2 + 32  \nu^{-1} \widehat{V}
|M_{k,r}|^{\frac{1}{\hat{q}}}
\\
&\leqslant  9\mu^2\nu^{-2} \|(\psin-k)\nabla\chi\|_{L_2(M_{k,r})}^2
+ 32\nu^{-1} \widehat{V}|M_{k,r}|^{\frac{1}{\hat{q}}},
\end{align*}
where we have used that $\frac{\mu}{\nu}\geqslant 1$. Now we take any
$\d\in(0,1)$ and assume that $\chi\equiv 1$ in $B_{r(1-\d)}(a)$ and
$|\nabla\chi|\leqslant 3(\d r)^{-1}$ in $B_r(a)$. Then we obtain
\begin{equation*}
\|\nabla \psin\|_{L_2(M_{k,r(1-\d)})}^2\leqslant  \Big(
81\Th_n^{\frac{2}{q}}\mu^2\nu^{-2}\d^{-2} r^{-2\left(1-\frac{n}{q}\right)}
\max_{\overline{M}_{k,r}}|\psin-k|^2 + 32\nu^{-1}\widehat{V}
\Big)|M_{k,r}|^{\frac{1}{\hat{q}}}.
\end{equation*}
This inequality means that the function $\psin$ belongs to the
aforementioned class
$\mathfrak{B}_2(B_{\frac{d}{4}}(\widetilde{a}),1,C,\widetilde{C},1,\frac{1}{q})$
with $C=81\Th_n^{\frac{2}{q}}\mu^2\nu^{-2}$,
$\widetilde{C}=32\nu^{-1}\widehat{V}$.

\s
Note that \cite[Theorem~6.1]{LadyzhenskayaU-73} implies that the estimate
(\ref{4.1}) holds. To obtain explicit expressions for the exponent
$\alpha$ and the constsnt $C_{14}$, one has to trace the
dependence of the various constants trough the proof of
\cite[Theorem~6.1]{LadyzhenskayaU-73}. More precisely, one uses formula (6.36) in
\cite[Chapter I\!I,Theorem~6.1]{LadyzhenskayaU-73}, the choice $\delta_0=\frac{1}{2}$  given
right after formula \cite[Chapter I\!I, (6.34)]{LadyzhenskayaU-73}, formulae
(6.24), (6.26) in the proof of \cite[Chapter I\!I, Lemma~6.3]{LadyzhenskayaU-73},
formulae (6.16), (6.17) in the proof of \cite[Chapter
I\!I, Lemma~6.2]{LadyzhenskayaU-73}, and the formula established at the very end of the
proof of \cite[Chapter I\!I, Lemma~3.8]{LadyzhenskayaU-73} for the constant $\b$
which is introduced in \cite[Chapter I\!I, Inequality~(3.4)]{LadyzhenskayaU-73}.
\qed\end{proof}

The final result of this section provides an upper bound   on
the $L_2$-norm of $\psin$, given the normalization (\ref{3.3}).

\begin{lemma}\label{lm4.4}
The estimate
\begin{equation}\label{4.6a}
\|\psin\|_{L_2(\Om)}^2\leqslant \frac{C_{18}}{|\lambdan|}, \qquad
C_{18}:=\left(C_1+\frac{4\mu}{r_1^2} \right)|\Omega_{0,\frac{d}{4}}|,
\end{equation}
holds true.
\end{lemma}

\begin{proof}
Let $\chi\in C^\infty(\overline{\Omega})$ be a cut-off function
vanishing in $\overline{\Omega}_0$ and equalling one in
$\Omega\setminus\Omega_{0,r_1}$. It is clear that $\psin\chi^2\in
\Dom(\mathfrak{h})$. In view of this fact and the definition of
$\psin$ we have
\begin{equation}\label{4.6b}
\mathfrak{h}[\psin,\psin\chi^2]=\lambdan\|\psin\chi\|_{L_2(\Om)}^2.
\end{equation}
Direct calculations using the symmetry of $A_{ij}$ yield
\begin{align*}
\widetilde{\mathfrak{h}}[\nabla\psin,\nabla\psin\chi^2]
&=\widetilde{\mathfrak{h}}[\chi\nabla\psin,\nabla\psin\chi] +
\widetilde{\mathfrak{h}}[\chi\nabla\psin,\psin\nabla\chi] \\
&=\widetilde{\mathfrak{h}}[\nabla\chi\psin,\nabla\chi\psin]
-\widetilde{\mathfrak{h}}[\psin\nabla\chi,\nabla\chi\psin]
+ \widetilde{\mathfrak{h}}[\psin\nabla\chi,\chi\nabla\psin] \\
&=
\widetilde{\mathfrak{h}}[\nabla\chi\psin,\nabla\chi\psin]\,
\widetilde{\mathfrak{h}}[\psin\nabla\chi,\psin\nabla\chi]
\end{align*}
We substitute this identity into (\ref{4.6b}) and obtain
\begin{equation*}
\mathfrak{h}[\psin\chi,\psin\chi]-\lambdan\|\psin\chi\|_{L_2(\Om)}^2
=\widetilde{\mathfrak{h}}[\psin\nabla\chi,\psin\nabla\chi].
\end{equation*}
The function $\psin\chi$ vanishes on $\Omega_0$ and this is why it
belongs to the domain of the quadratic form associated with
$\mathcal{H}_0$. The value of this quadratic form on $\psin\chi$
equals $\mathfrak{h}[\psin\chi,\psin\chi]$, and by
assumptions~(\ref{2.3-c}) and (\ref{2.3-d}) we obtain that $\lambdan<0$ and
$\mathfrak{h}[\psin\chi,\psin\chi]\geqslant 0$.
These inequalities and (\ref{1.4}) imply
$
|\lambdan|\|\psin\chi\|_{L_2(\Om)}^2\leqslant
\mu\|\psin\nabla\chi\|_{L_2(\Om)}^2.
$
We choose $\chi$ so that it takes values in $[0,1]$ and satisfies
$|\nabla\chi|\leqslant 2r_1^{-1}$ on $\Omega_{0,r_1}\setminus\Omega_0$.
Hence,
\begin{equation}
\begin{aligned}
\|\psin\|_{L_2(\Omega\setminus\Omega_{0,r_1})}^2&\leqslant
\frac{4\mu}{r_1^2|\lambdan|}\|\psin\|_{L_2(\Omega_{0,r_1})}^2, 
\\
\|\psin\|_{L_2(\Om)}^2&\leqslant |\lambdan|^{-1}
\left(|\lambdan|+\frac{4\mu}{r_1^2}\right)
\|\psin\|_{L_2(\Omega_{0,r_1})}^2.\label{4.10}
\end{aligned}
\end{equation}
 We know by (\ref{6.3c}) that $|\psin|\leqslant 1$   in \vspace{.6mm}
$\Omega_{0,r_1}$. It also follows from the definition of $r_1$ that
$|\Omega_{0,r_1}|\leqslant |\Omega_{0,\frac{d}{4}}|$. Substituting these estimates
into (\ref{4.10}) and employing the inequality $|\lambdan|< C_1$ \vspace{.6mm}
which is valid due to (\ref{2.4}), we arrive at the statement of the
lemma.
\qed\end{proof}

\section{Proof of main results}
\label{s:Proof of main results}

\begin{proof}[Proof of Theorem~{\rm\ref{th1.1}}]

Let $x_+\in\overline{\Omega}_0$ be the point where the function $\psin$
attains the maximum, i.e.,  $\psin(x_+)=1$.  Such a point exists due
to (\ref{3.3}). Due to Lemma~\ref{lm4.2}, there exist a point $x_-$
in $\Omega_0$ such that $\psin(x_-)=0$.   We connect the points $x_+$
and $x_-$ by an admissible cylinder and choose it as the domain
$\Omega'$. We fix the radius of the cylinder $\Omega'$ setting it equal to
\begin{equation*}
r_2:=\min\left\{ r_1(3C_{14}C_2)^{- \frac{1}{\a}}, \frac{d}{8},r_0 \right\},
\end{equation*}
where the constant $C_2$ is taken from Theorem \ref{th:Harnack}
(Harnack inequality). We denote the
bases of $\Omega'$ by $S_+,S_-$ (so that $x_+\in S_+$ and $x_-\in
S_-$). In this cylinder we introduce new coordinates: the arc
length $s$ of the $C^2$-curve corresponding to $\Omega'$ connecting
$x_+$ and $x_-$, and the coordinates on the cross-section. Since
the cylinder is admissible, these coordinates are well-defined. We
also observe that $\Omega'\subseteq\widehat{\Omega}_{\frac{d}{8}}$.

Lemma~\ref{lm4.1} and the definition of $r_2$ imply that
\begin{equation*}
\psin(x)\geqslant 1-(3C_2)^{-1} \ \text{for } x\in S_+ \quad
\text{and}\quad \psin(x)\leqslant (3C_2)^{-1}\ \text{for } x\in S_-.
\end{equation*}
We employ this inequality and the obvious estimate
$
\Big|\frac{\p\psin}{\p s}\Big|\leqslant |\nabla\psin|,$
$x\in\Omega',
$
to obtain
\begin{equation}
\begin{aligned}
\inf_{\Omega'}\,\psi_0\cdot\Big\|\nabla
\frac{\psin}{\psi_0}\Big\|_{L_1(\Omega')}&\geqslant
\inf_{\widehat{\Omega}_{\frac{d}{8}}}\psi_0\cdot\int_{\Omega'}\frac{\p}{\p
s}\frac{\psin}{\psi_0}\di x
\\
&=
\inf_{\widehat{\Omega}_{\frac{d}{8}}}\,\psi_0 \cdot\left(
\int_{S_+}\frac{\psin}{\psi_0}\di
S_+-\int_{S_-}\frac{\psin}{\psi_0}\di S_-\right)
\\
&\geqslant S_0\left(
\frac{\inf_{\widehat{\Omega}_{\frac{d}{8}}}\psi_0}
{\sup_{\widehat{\Omega}_{\frac{d}{8}}}\psi_0}
\left(1-(3C_2)^{-1}\right)-(3C_2)^{-1}\right)
\\
&\geqslant
\frac{S_0(2C_2-1)}{3C_2^2}\\
&\geqslant \frac{S_0}{3C_2},
\end{aligned}
\label{3.6}
\end{equation}
where $S_0:=|S_+|=|S_-|$.  Recall that the cylinder $\Omega'$ is
defined with the help of a curve connecting the points $x_-$ and
$x_+$. Let $\ell$ be the length of the $C^2$-curve connecting $x_-$
and $x_+$. To  estimate the volume of $\Omega'$ we will need the
following auxiliary
\begin{lemma}\label{lm8.1}
The equality $|\Omega'|=S_0 \ell$ holds true.
\end{lemma}
\begin{proof}
Let $R(s)$ be the vector-function describing the $C^2$-curve
connecting $x_-$ and $x_+$, where $s$ is the arc length, $T(s)$ be
the tangential vector to this curve, and $N_i(s)$, $i=1,\ldots,n-2$
be the continuously differentiable vectors orthogonal to~$T(s)$. We
assume that $N_i$ are orthonormalized, so, the vectors $T$ and
$N_i(s)$ form a Frenet frame attached to the curve. The vectors
$N_i$ form an orthonormalized basis in the $(n-1)$-dimensional disk
attached to the same point of the curve as~$N_i$. As the
corresponding Cartesian coordinates $y$ we choose the ones
associated with the vectors $N_i$.  As a result we have
$$
x=R(s)+\sum_{i=1}^{n-1} y_i N_i(s).
$$
By $J(s,y)$ we denote the Jacobian
\begin{equation*}
J(s,y)=\frac{D(x)}{D(s,y)}=\det M_J, \quad \text{where}\quad M_J:=
\left(
\begin{matrix}
T(s)+\sum_{i=1}^{n-1} y_i N_i'(s)
\\
N_1(s)
\\
\vdots
\\
N_{n-1}(s)
\end{matrix}\right).
\end{equation*}
Since the vectors $T$ and $N_i$ are orthonormalized, the matrix
$(T,N_1,\ldots,N_{n-1})$ is unitary and up to a renumbering of $N_i$
we can assume that its determinant equals one. Hence, if we multiply
$M_J$ by this matrix, we do not change the value of $J$. It  gives
\begin{equation}\label{8.1}
J(s,y)=\det\left(
\begin{matrix}
1+\sum_{i=1}^{n-1} y_i k_i(s) & * & * & \ldots & *
\\
0 & 1 & 0 & \ldots & 0
\\
0 & 0 & 1 & \ldots & 0
\\
\vdots & \vdots & \vdots & \ddots & \vdots
\\
0 & 0 & 0 & \ldots & 1
\end{matrix}\right)=1+\sum_{i=1}^{n-1} y_i k_i(s),
\end{equation}
where the symbol $*$ indicates unspecified  functions,
$k_i(s)=(T(s),N_i'(s))_{\mathbb{R}^n}$. Since
for each $s$ the vectors $(T,N_1,\ldots,N_{n-1})$ form a basis
 the map $ s \mapsto J(s,y)$ never vanishes.
As $J(s,y)=1$ for  $y=0$, we conclude that
$J(s,y)$ is a positive function. Employing this fact and
(\ref{8.1}), we can calculate the volume of $\Omega'$:
\begin{equation*}
|\Omega'|=\int_0^\ell \di s \int_{|y|<r_2} J(s,y)\di
y=|S_0|\ell +\int_0^\ell \di s \sum_{i=1}^{n-1} k_i(s)
\int_{|y|<r_2} y_i\di y.
\end{equation*}
By parity arguments
$
\int_{|y|<r_2} y_i\di y=0.
$
Together with the previous identity this completes the proof of the auxiliary lemma.
\qed\end{proof}

Now we continue the proof of Theorem~\ref{th1.1}. For this purpose
we substitute the proved identity $|\Omega'|=S_0 \ell$
and (\ref{3.6})into (\ref{3.2}), and arrive at the estimate
\begin{equation}\label{3.7}
\lambdan-\lambda_0\geqslant \frac{S_0^2\nu}{9C_2^2|\Omega'|
\|\psin\|_{L_2(\Om)}^2}\geqslant \frac{S_0\nu}{9C_2^2
L\|\psin\|_{L_2(\Om)}^2},
\end{equation}
where we have used that $\ell\leqslant L$. The $(n-1)$-dimensional
  volume of the discs $S_+,S_-$ equals
$
S_0=\Th_{n-1}r_2^{n-1}.
$
We substitute this identity and (\ref{4.6a}) into (\ref{3.7}) and
arrive at (\ref{1.6}). In this inequality we changed the notations,
namely, we denoted $c_1:=r_2$,  $c_2:=C_{13}$, $c_3:=C_{14}$,
$c_4:=C_{15}$, $c_5:=C_{16}$, $c_6:=C_{17}$, $c_7:=C_{10}$,
$c_8:=C_5  \frac{C_7}{p}$, $c_9:=C_{12} \frac{\hat{q}}{p}$.  
\qed\end{proof}

\begin{proof}[Proof of Theorem~{\rm\ref{th2.1}}] Let us prove the estimate
(\ref{1.7}). The parameter $L$ appears only explicitly in the right
hand side of (\ref{1.6}), and also in the definition of $c_1$. We
also observe that $|\Omega_{0,\frac{d}{4}}|$ is bounded by  $\Theta_n L^n$. The
constants $c_i$, $i=2,\ldots,10$ are independent of $L$. It follows \vspace{-.4mm}
from (\ref{5.17}) and the formula for $c_3$ that $c_2>1$, $c_3>1$.
Hence, for $L$ large enough $c_1=r_1(3c_3 c_2^{\frac{8L}{d}})^{- \frac{1}{\a}}$. We \vspace{.6mm}
substitute this identity into~(\ref{1.6}) and arrives at
(\ref{1.7}), where $c_{11}:=8d^{-1}(1+(n-1)\a^{-1})\ln c_2>0$.
\qed\end{proof}

\begin{proof}[Proof of Theorem~{\rm\ref{th2.5}}]
It is clear that for $\widehat{V}$ small enough all the constants
remain bounded from above and below.  We also note that \vspace{-.8mm}
$\|V^-\|_{L_{\frac{q}{2}}(\Omega_0)}$ is small, too. The mentioned facts imply
(\ref{1.11}).
\qed\end{proof}

\begin{proof}[Proof of Theorem~{\rm\ref{th2.2}}] In the case considered the
constants $\a$, $c_3$--$c_6$
 remain constant and
depend on $n$, $q$, and $\frac{\mu}{\nu}$. The constants $c_7$--$c_{10}$
 satisfy the relations
\begin{equation}\label{6.6}
c_7^{-1}\leqslant C_{19}\nu^{-\frac{1}{2}}, \quad c_8\leqslant
C_{20}\nu^{-2
}, \quad c_9\leqslant
C_{21}\nu^{-1}. 
\end{equation}
where $C_i=C_i(n,q,d,\widehat{V})$, $i=19,20,21$. Hence,
\begin{equation*}
c_2\leqslant C_{22}\nu^{-\frac{\hat{q}}{p-\hat{q}}}\big(C_{23}
\log_{\frac{p}{\hat{q}}}^2\nu\big)^{C_{24}\nu^{-\frac{1}{2}}}\!\!, \quad
r_1=C_{25}\nu^{\frac{q}{2(q-n)}}\!,
\quad c_1=
C_{26}\nu^{\frac{q}{2(q-n)}}c_2^{-\frac{8L}{\a d}}\!,
\end{equation*}
where $C_i=C_i(n,q,d,\widehat{V})$, $i=22,\ldots,26$, $C_{24}>0$,
$C_{25}\not=0$, $C_{26}\not=0$. We substitute these relations into
(\ref{1.6}) and obtain (\ref{1.8}).
\qed\end{proof}

\begin{proof}[Proof of Theorem~{\rm\ref{th2.3}}]
The proof of (\ref{1.9}) is more complicated in comparison with the
previous proof. Namely, in this case the estimates (\ref{6.6}) for $c_7$, $c_8$,
remain true, where ${C}_i={C}_i(n,q,d,\mu,\widehat{V})$.
The estimates for $c_9$ and $c_2$ read as follows,
\begin{equation*}
c_9\leqslant C_{27}\nu^{-1}\log_{\frac{p}{\hat{q}}}^2 \nu,
c_2\leqslant (C_{28}\nu^{-1}\log_{\frac{p}{\hat{q}}}^2 \nu)^{C_{29}\nu^{-\frac{1}{2}}},
\end{equation*}
${C}_i={C}_i(n,q,d,\mu,\widehat{V})$, $i=27,28,29$.
The main difference
with the previous case is that now the constants $c_3$--$c_6$
and $\a$ depend on $\nu$ in a singular way. Namely,
\begin{align*}
c_6\leqslant C_{30}\nu^{-1}, \quad c_5\leqslant
C_{31}\nu^{-\frac{qn}{q-n}},\quad c_4\leqslant
C_{32}\nu^{-\frac{2(n-1)q}{q-n}}.
\end{align*}
where ${C}_i={C}_i(n,q,\mu)$, $i=30,31,32$. Thus,
$\a=-\log_4(1-2^{-c_4})$, and
$$
C_{33}\exp\!\Big(\!\!-C_{34}\nu^{-\frac{2(n-1)q}{q-n}}\!\Big) \!\geqslant \!\a\!
\geqslant
\!C_{35}\exp\!\Big(\!\!-C_{34}\nu^{-\frac{2(n-1)q}{q-n}}\!\Big)\!,
\ \
c_3\!\leqslant\! C_{37}\exp\!\Big(\!C_{38}\nu^{-\frac{2(n-1)q}{q-n}}\!\Big)\!,
$$
where ${C}_i={C}_i(n,q,\mu)$, $i=34,36,37$, $C_{34}>0$, $C_{38}>0$,
$C_{33}$ and $C_{35}$ are some absolute constants. We also observe \vspace{1mm}
that  $r_1\sim \nu^{\frac{q}{2(q-n)}}$,
$c_1=r_1(3c_3c_2^{\frac{8L}{d}})^{- \frac{1}{\a}}$. Bearing in mind the obtained \vspace{1mm}
relations, we estimate the right hand side of (\ref{1.6}) from below
that gives the following inequality
\begin{align*}
\l-\lambda_0&\geqslant |\Omega_{0,\frac{d}{4}}|^{-1}C_{39}|\l|L^{-1}
\nu c_1^{n-1} r_1^2  c_2^{-\frac{8L}{d}}\\
&\geqslant C_{39}|\l|L^{-n-1} \nu
r_1^{n+1} 3^{-\frac{n-1}{\a}} c_3^{-\frac{n+1}{\a}}
c_2^{-\frac{8Ln}{\a d}}
\\
&\geqslant C_{40}|\l|L^{-n-1}\nu^{1+\frac{q(n+1)}{2(q-n)}}
\exp\left(\hspace*{-3pt}-C_{41}\a^{-1} \nu^{-\frac{2(n-1)q}{q-n}}\right)
\big(C_{28}\nu^{-1}
\log_{\frac{p}{\hat{q}}}^2\nu\big)^{-\frac{8{C}_{24}Ln}{\a d
\sqrt{\nu}}} \!,
\end{align*}
where $C_{39}=C_{39}(n,d,\mu)$,
$C_{40}=C_{40}(n,d,\mu,\widehat{V})$, $C_{41}=C_{41}(n,q,\mu)$.
We substitute into the obtained inequality the estimate for $\a$
that implies (\ref{1.9}).
\qed\end{proof}

\subsection*{Acknowledgments}
D.B. was partially supported by RFBR, by the Federal Task Program of the Ministry of Education and Science of Russia (contract no.~02.740.11.0612), and by FCT, project PTDC/\ MAT/\ 101007/2008. Both authors were partially supported through the project
``Spektrale Eigenschaften von zuf\"alligen Schr\"odingeroperatoren und zuf\"alligen Operatoren auf Mannigfaltigkeiten und Graphen''
within the Emmy-Noether-Programme of the Deutsche Forschungsgemeinschaft.


\begin{thebibliography}{99}


\bibitem{Borisov-07}
Borisov,~D.,
Asymptotic behaviour of the spectrum of a waveguide with distant perturbations.
{\it Math. Phys. Anal. Geom.} 10  (2007), 155 -- 196.

\bibitem{Borisov-07a}
Borisov,~D.,
Distant perturbations of the Laplacian in a multi-dimensional space.
{\it Ann. Henri Poincar\'e} 8 (2007), 1371 -- 1399.

\bibitem{Davies-89}
Davies,~E.~B.,
{\it Heat Kernels and Spectral Theory.}
Cambridge: Cambridge University Press 1989.

\bibitem{DuclosE-95}
Duclos,~P. and Exner,~P.,
Curvature-induced bound states in quantum waveguides in two and three dimensions.
{\it Rev. Math. Phys.} 7  (1995),  73 -- 102.

\bibitem{GilbargT-77}
Gilbarg,~D. and Trudinger,~N.,
{\it Elliptic Partial Differential Equations of Second Order.}
Berlin: Springer 1977.

\bibitem{Harrell-78}
Harell,~E.,
On the rate of asymptotic eigenvalue degeneracy.
{\it Comm. Math. Phys.}  60 (1978), 73 -- 95.


\bibitem{KirschS-85}
Kirsch,~W. and Simon,~B.,
Universal lower bounds of eigenvalue splittings for one dimensional {Schr\"odinger} operators.
{\it Comm. Math. Phys.}  97 (1985), \lb
453 -- 460.

\bibitem{KirschS-87}
Kirsch,~W. and Simon,~B.,
Comparison theorems for the gap of  Schr\"odinger  operators.
{\it J. Funct. Anal.}  75 (1987), 396 -- 410.

\bibitem{KolmogorovF-75}
Kolmogorov,~A.~N. and Fomin,~S.~V.,
{\it Introductory Real Analysis.}
Transl. from second Russian edition, ed.: R.~A.~Silverman, corrected reprinting.
New York: Dover 1975.

\bibitem{KondejV-06a}
Kondej,~S. and Veseli\'c,~I.,
Lower bounds on the lowest spectral gap of singular potential Hamiltonians.
{\it Ann. Henri Poincar\'e} 8 (2007), 109 -- 134.


\bibitem{Kuchment-01}
Kuchment,~P.,
The mathematics of photonic crystals.
In: {\it Mathematical Modeling in Optical Science} (eds.: G.~Bao et al.). Frontiers Appl.~Math. 22. Philadelphia (PA): SIAM 2001,  pp. 207 -- 272.



\bibitem{LadyzhenskayaU-73}
Ladyzhenskaya,~O.~A. and Uraltseva,~N.~N.,
{\it Lineinye i Kvazilineinye Uravneniya Ellipticheskogo Tipa} (in Russian).
Second rev.~edition. Moscow: Nauka 1973.


\bibitem{MaL-08}
Ma,~L. and Liu,~B.,
Convexity of the first eigenfunction of the drifting Laplacian operator and its applications.
{\it New York J. Math.} 14 (2008), 393 -- 401.

\bibitem{Simon-84a}
Simon,~B.,
Semiclassical analysis of low lying eigenvalues. II. Tunneling.
{\it Ann.~Math.}  120 (1984), 89 -- 118.

\bibitem{Simon-84b}
Simon,~B.,
Semiclassical analysis of low lying eigenvalues.  III. {W}idth of the ground state band in strongly coupled solids.
{\it Ann. Physics} 158 (1984), \lb
415 -- 420.

\bibitem{Simon-85d}
Simon,~B.,
Semiclassical analysis of low lying eigenvalues. IV. The flea on the elephant.
{\it J. Funct. Anal.} 63 (1985), 123 -- 136.

\bibitem{SingerWYY-85}
Singer,~I.~M., Wong,~B., Yau,~S.-T. and Yau,~S.~S.-T.,
An estimate of the gap of the first two eigenvalues in the Schr\"odinger operator.
{\it Ann. Scuola Norm. Sup. Pisa Cl. Sci. (4)} 12 (1985), 319 -- 333.

\bibitem{Vogt-09}
Vogt,~H.,
A lower bound on the first spectral gap of schr\"odinger operators with Kato class measures.
{\it Ann. Henri Poincar\'e} 10 (2009), 395 -- 414.

\bibitem{Yau-03}
Yau,~S.-T.,
An estimate of the gap of the first two eigenvalues in the {S}chr\"odinger operator.
In: {\it Lectures on Partial Differential Equations} (eds.: S.-Y.~A.~Chang et al.).
New Stud. Adv. Math. 2. Somerville (MA): Int. Press 2003, pp. \lb
223 -- 235.

\bibitem{Yau-09}
Yau,~S.-T.,
Gap of the first two eigenvalues of the {S}chr\"odinger operator with  nonconvex potential.
{\it Mat.~Contemp.}
35 (2008), 267 -- 285.

\end{thebibliography}
\end{document}